\documentclass[12pt]{extarticle}

\usepackage[utf8]{inputenc}
\usepackage{caption}
\oddsidemargin \evensidemargin
\usepackage{stmaryrd}
\usepackage{floatrow}
\usepackage{url}
\usepackage{multicol}
\usepackage{amssymb, amsfonts, amsthm}
\usepackage{xfrac}
\usepackage{enumerate}
\usepackage{verbatim}
\usepackage[all]{xy}
\usepackage[utf8]{inputenc}
\usepackage{bbm}
\usepackage{mathrsfs}
\usepackage{mathtools}
\usepackage{listings}
\usepackage{maple}
\usepackage{amsmath}
\usepackage{breqn}
\usepackage{textcomp}
\usepackage{moreverb}
\usepackage{verbatim}
\usepackage{sverb}
\usepackage{algorithm}
\usepackage{algpseudocode}
\usepackage{algorithmicx}
\usepackage{euscript}
\usepackage{subcaption}

\renewenvironment{thebibliography}[1]{
  \begin{oldthebibliography}{#1}
    \setlength{\itemsep}{0.6em}
    \setlength{\parskip}{0em}
}
{
  \end{oldthebibliography}
}

\usepackage[dvipsnames]{xcolor}
\usepackage[colorlinks=true, allcolors=MidnightBlue,linkcolor=MidnightBlue, urlcolor=black,citecolor=MidnightBlue,backref=page]{hyperref}

\usepackage{cleveref}

\theoremstyle{definition}
\newtheorem{Algorithm}{Algorithm}

\numberwithin{equation}{section}
\newtheorem{theorem}{Theorem}[section]

\newtheorem{lemma}[theorem]{Lemma}
\newtheorem{corollary}[theorem]{Corollary}
\newtheorem{proposition}[theorem]{Proposition}

\theoremstyle{definition}

\newtheorem{definition}[theorem]{Definition}

\newtheorem{examplex}[theorem]{Example}
\newenvironment{example}
{\pushQED{\qed}\examplex}
{\popQED\endexamplex}

\newtheorem{remarkx}[theorem]{Remark}
\newenvironment{remark}
{\pushQED{\qed}\remarkx}
{\popQED\endremarkx}

\newtheoremstyle{citing}% name
{}%      Space above, empty = `usual value'
{}%      Space below
{\itshape}% Body font
{}%         Indent amount (empty = no indent, \parindent = para indent)
{\bfseries}% Thm head font
{\textbf{.}}%        Punctuation after thm head
{.5em}%     Space after thm head: " " = normal interword space;
{\thmnote{#3}}% Thm head spec
{\theoremstyle{citing}
}

\newcommand{\NN}{\mathbb{N}}

\newcommand{\RR}{\mathbb{R}}

\newcommand{\QQ}{\mathbb{Q}}
\newcommand{\KK}{\mathbb{K}}

\DeclareMathOperator{\Spec}{Spec}
\DeclareMathOperator{\Span}{span}

\DeclareMathOperator{\ord}{ord}

\DeclareMathOperator{\res}{res}

\newcommand{\FF}{\mathbb{F}}

\newcommand{\cA}{\mathcal{A}}

\newcommand{\tr}{\mathrm{trdeg}}
\newcommand{\cF}{\mathcal{F}}

\title{D-Algebraic Functions}
\author{\normalsize Rida Ait El Manssour, Anna-Laura Sattelberger, Bertrand Teguia Tabuguia}
\date{}

\begin{document}

%for maple
\DefineParaStyle{Maple Bullet Item}
\DefineParaStyle{Maple Heading 1}
\DefineParaStyle{Maple Warning}
\DefineParaStyle{Maple Heading 4}
\DefineParaStyle{Maple Heading 2}
\DefineParaStyle{Maple Heading 3}
\DefineParaStyle{Maple Dash Item}
\DefineParaStyle{Maple Error}
\DefineParaStyle{Maple Title}
\DefineParaStyle{Maple Ordered List 1}
\DefineParaStyle{Maple Text Output}
\DefineParaStyle{Maple Ordered List 2}
\DefineParaStyle{Maple Ordered List 3}
\DefineParaStyle{Maple Normal}
\DefineParaStyle{Maple Ordered List 4}
\DefineParaStyle{Maple Ordered List 5}
\DefineCharStyle{Maple 2D Output}
\DefineCharStyle{Maple 2D Input}
\DefineCharStyle{Maple Maple Input}
\DefineCharStyle{Maple 2D Math}
\DefineCharStyle{Maple Hyperlink}

\setlength\parindent{14pt}

%%%%

\thispagestyle{empty}
\maketitle 

\begin{abstract}
Differentially-algebraic (D-algebraic) functions are solutions of polynomial equations in the function, its derivatives, and the independent variables. We revisit closure properties of these functions by providing constructive proofs. We present algorithms to compute algebraic differential equations for compositions and arithmetic manipulations of univariate \mbox{D-algebraic} functions and derive bounds for the order of the resulting differential equations. We apply our methods to examples in the sciences. 
\end{abstract}

\section{Introduction}\label{sec:intro}
The Weyl algebra encodes linear differential operators with polynomial coefficients, such as the operator $P=\partial^2-x$ arising from Airy's differential equation
\vspace*{-1mm}
\begin{align}\label{weierstrasseq}
f''(x)-xf(x) \,=\, 0 \, .
\end{align}
\vspace*{-6mm}

\noindent Solutions of such differential equations are called {\em holonomic} functions. Differential algebra investigates {\em polynomials} in a differential indeterminate $y$ and its derivatives with coefficients in a differential ring. Those polynomials are called {\em differential polynomials} and their associated differential equations are commonly referred to as {\em algebraic differential equations} (ADEs).  For instance, the differential polynomial $p=(y')^2-4y^3-g_2y-g_3,$ where $g_2$ and~$g_3$ are complex constants, encodes the differential equation 
\begin{align}
\left(y'(x)\right)^2 \,=\, 4\left(y(x)\right)^3 + g_2y(x) + g_3 \, ,
\end{align} 
of which the Weierstrass elliptic function $\wp$ (see \cite{reinhardt2009weierstrass, jones1987complex}) is a solution of. Among many more fields of application, ADEs naturally arise in structural identifiability~\cite{DGHP}.
Functions which are zeros of differential polynomials---or, equivalently, solutions of the corresponding ADEs---are called {\em D-algebraic functions}. This class of functions arises in a natural way from the study of holonomic functions: for instance, the reciprocal of a holonomic function is in general not holonomic, but it is \mbox{D-algebraic}. The present article can hence be located at the interim of the theory of $D$-ideals and differential algebra. For introductions to those fields, we refer our readers to \cite{SatStu19,SST00} and \cite{Ritt, kaplansky1976introduction}, respectively. 
Although they are present at several places in the literature (see \cite{rubel1983some,sibuya1981arithmetic,rubel1992some,katriel2003solution}), it is difficult to find exhaustive expositions about D-algebraic functions. This could be explained by the challenging computational complexity observable from \cite{hebisch2011extended}, which is detrimental for applications. One attempt of computer algebra is to find subclasses of D-algebraic functions that offer both a mathematical structure and efficient algorithms for the corresponding arithmetic. Some relevant subclasses are holonomic functions, which are also called ``D-finite functions'' (see \cite{HolFun, stanley1980differentiably}), DD-finite functions that satisfy differential equations with D-finite coefficients \cite{jimenez2019computable}, or 
functions that satisfy ADEs of degree at most~$2$ \cite{TeguiaDelta2, TBguessing}, called ``$\delta_2$-finite functions'' therein. While the rich theory of $D$-modules and holonomic functions covers the degree-one case, similar constructions for ADEs of higher degree are still in an early stage. The idea of bounding the degree of the ADEs is to alleviate the complexity compared to the general~case.
The class of D-algebraic functions has nice closure properties, which is for instance carried out in \cite{raschel2020counting} and \cite{vdH19}. 
\cite{raschel2020counting} illustrates the use of D-algebraic functions in combinatorics for studying generating functions, with an emphasis on quadrant walks via Tutte's invariant method. 
\cite{vdH19} focuses on the zero test problem for D-algebraic functions viewed as formal power series. The earliest appearance of the terminology ``differentially algebraic'' that we are aware of is in Rubel's~work~\cite{rubel1992some}.

In this article, we construct ADEs for compositions and rational functions of D-algebraic functions, taking only the ADEs of the original functions as input. We give bounds for the order of the resulting differential polynomials. Finding an ADE fulfilled by a rational expression in a D-algebraic function $f$, or finding an ADE fulfilled by its antiderivative, are unary operations in the input ADE. The rest of operations---like sums, products, ratios, and compositions---are {\em binary} operations. However, we implemented them for arbitrarily many operands. 
Our algorithms take differential polynomials $p,\, q$ as input. From those, we construct ADEs which rational functions, antiderivatives, and compositions of all solutions of the input ADEs fulfill. Throughout, we tacitly assume that the respective domains of the functions are compatible. We stress that working with the differential polynomials implies working with the set of {\em all} solutions of the corresponding ADEs.

Gr\"{o}bner basis theories exist both for $D$-ideals and for differential ideals. Viewing ADEs as differential polynomials, one can encode an operation between D-algebraic functions by an ideal in a differential polynomial ring. Such an ideal is differentially generated by the given polynomials and a rational expression built from the underlying operation. We use jets to truncate the obtained differential ideal at the desired order before applying elimination theory based on Gr\"{o}bner basis techniques. 
We compare it to a second method that still proceeds with Gr\"{o}bner bases, but uses bounds for the order of the resulting differential polynomial, which we establish in our study.
The main contribution of this article is the development of two strategies to compute differential equations satisfied by rational expressions and compositions of \mbox{D-algebraic} functions. In particular, our algorithms and their implementations are general, reliable, and outperform existing algorithms and software, which often do not contain such computations. For instance, the {\tt find_ioequations} command of the Julia~\cite{Julia} package {\tt StructuralIdentifiability} can be used to derive ADEs with constant coefficients only; moreover, that command requires a user-defined dynamical model, which is not needed in our algorithms for arithmetic operations and compositions. For each operation with \mbox{D-algebraic} functions, we provide a bound for the order of the resulting differential polynomial in Theorems~\ref{thm:boundldf} and \ref{thm:cboundldf} which depends on those of the input ADEs. We implemented our first method in Macaulay2~\cite{M2}, since it is open source and it is one of the systems commonly used by algebraic geometers. The second one is implemented in the Maple~\cite{maple} package {\tt NLDE}~\cite{teguiaoperations}. Our implementations are made available via the MathRepo~\cite{mathrepo}---a repository website hosted by MPI~MiS---at \url{https://mathrepo.mis.mpg.de/DAlgebraicFunctions}.

\pagebreak
{\bf Outline.}
Our article is organized as follows. \Cref{sec:AlgODE} recalls basic concepts about the Weyl algebra as well as differential algebra.
In \Cref{sec:dalgclosure}, we present D-algebraic functions and investigate their closure properties. In \Cref{sec:dalgfunc}, we study the arithmetic of these functions: we construct ADEs for antiderivatives, compositions and rational expressions of D-algebraic functions and study the order of the resulting differential equations. \Cref{sec:algo} presents pseudocode to carry these operations out in practice. We present applications of our results in the study of Feynman integrals and epidemiology. The appendix recalls basic notions about algebraic varieties and jet schemes.
 
\bigskip

{\bf Notation.} By $\NN$, we denote the nonnegative integers, and by $\KK$ a field of characteristic~$0$. Letters $f,g,h$ are reserved for functions, $D$ denotes the Weyl algebra, $P\in D$ is used for linear differential operators. Differential polynomials are denoted by lower case letters $p,q,r,s$. They are polynomials in dependent variables $u,y,z,w$ and the independent variable $x$ (or $x_1,\ldots, x_k$ in the multivariate~case). ``Degree'' means the total degree of polynomials, unless stated otherwise. By a {\em rational expression in $f$}, we will mean a rational function in $x$, $f$, and the derivatives of $f$.

\section{Algebraic aspects of differential equations} \label{sec:AlgODE}
We briefly recall some algebraic aspects of differential equations. We here revisit two classes of differential equations, namely linear and algebraic differential equations with polynomial coefficients. Throughout this section, we describe the case of univariate functions.

\subsection{The Weyl algebra and holonomic functions}\label{sec:HolFun}
The Weyl algebra, denoted $D\coloneqq \KK[x]\langle \partial \rangle,$ is the free $\KK$-algebra generated by $x$ and $\partial$ modulo  the following relation: the commutator of $\partial$ and $x$ is 
\begin{align}\label{eq:Leibniz}
\left[ \partial, x \right] \, = \, \partial x - x \partial  \, = \, 1 \, ,
\end{align}
which encodes Leibniz' rule for taking the derivative of a product of functions in a formal~way. Hence, the Weyl algebra gathers linear differential operators with polynomial coefficients,~i.e.,
\begin{align}
D \, = \, \left\{ \sum_{i=0}^k a_i \partial^i \mid k\in \NN, \ a_i \in \KK[x]  \right\}.
\end{align}
The {\em order} of a differential operator $P=\sum_i a_i \partial^i$ is 
\begin{align} 
\ord (P) \, \coloneqq \, \max \left\{ i \mid a_i \neq 0\right\} \, .
\end{align}
For $P\in D,$ the vanishing locus of the leading polynomial $a_{\ord(P)}$ is the {\em singular locus} of $P.$ 
To a differential operator $P,$ one associates the ODE $P(f) = 0,$ i.e., one looks for functions~$f$ that are annihilated by the differential operator~$P.$

\begin{definition}
A function $f(x)$ is {\em holonomic} if there exists a differential operator $P\in D$ which annihilates $f$.
\end{definition}
Holonomicity of $D$-modules dates back to Bernstein and Kashiwara. Effective computations with holonomic {\em functions} were first studied by Zeilberger \cite{Zei90} for proving identities between special functions automatically. The name ``D-finite'' is justified by the following observation: a function $f$ is holonomic if and only if the $\KK(x)$-vector space spanned by the derivatives of $f$ is finite-dimensional, i.e., if 
\begin{align}\label{eq:Dfinitevs}
\dim_{\KK(x)} \left( \Span_{\KK(x)}\left( \left\{ \partial^k(f) \right\}_{k\in \NN} \right)\right)  \,< \,\infty \,.
\end{align}
In general, a (uni- or multivariate) function is called {\em holonomic} if its annihilating \mbox{$D$-ideal} is holonomic. We refer to \cite{SST00,SatStu19} for an introduction to $D$-modules with a focus on computational aspects and~applications.
Holonomic functions are ubiquitous in the sciences. Examples of holonomic functions include many special functions like error functions, Bessel functions, generalized hypergeometric functions, and linear combinations of elementary functions.
Plenty of computer algebra systems contain libraries for computations around $D$-ideals and holonomic functions, such as the {Mathematica} packages {\tt GeneratingFunctions} and {\tt HolonomicFunctions}~\cite{HolFun}, the package {\tt ore\_algebra} in {SAGE}, the built-in {\tt DEtools-FindODE} in {Maple} which incorporates {\tt HolonomicDE} from the package {\tt FPS}~\cite{FPS}, the package {\tt Dmodules.m2}~\cite{Dmodm2} in {Macaulay2}~\cite{M2}, and the $D$-module libraries~\cite{ABLMS} in {\sc Singular}~\cite{Singular}, just to name a few. The class of holonomic functions is well-behaved: it is closed under addition, multiplication, taking integrals and derivatives, and convolution---whenever defined---and some more operations. However, it is for instance not closed under taking compositions or reciprocals.
In order to decide if the reciprocal of a univariate function is holonomic, one can make use of the following characterization from~\cite{harris1985reciprocals}.

\begin{proposition}\label{prop:reciproc}
Let $f$ be holonomic. Its reciprocal $1/f$ is holonomic iff $f'/f$ is algebraic.
\end{proposition}

\begin{example}\label{ex:recnothol}
Let $f=\cos.$ Clearly, $f$ is holonomic of order~$2,$ since $f''+f=0.$ To that ODE, one associates the differential operator $P=\partial^2+1\in D.$  Its reciprocal $g=1/f$ is not holonomic, since $g$ has infinitely many poles. Indeed, in accordance with \Cref{prop:reciproc}, the function $f'/f=-\sin/\cos=-\tan$ is not algebraic. 
But one can compute a quadratic differential equation for $g$ by hand, or by using {\tt FPS:-QDE} of~\cite{TeguiaDelta2}, for instance. One finds the second-order ODE 
\begin{align} 
g(x)g''(x) - 2 {\left( g'(x)\right) }^2-\left( g(x)\right)^2 \,=\, 0
\end{align}
of degree~$2$.
\end{example} 
\Cref{ex:recnothol} also is an example of a non-holonomic composition $f_1\circ f_2$ of the holonomic functions $f_1(x)=1/x$ and $f_2(x)=\cos(x).$
Those considerations motivate to pass on to differential equations of higher degree.

\subsection{Differential algebra}\label{sec:diffalg}
Differential algebra studies differential equations that express a polynomial relation between a function and its derivatives. In this subsection, we recall basic definitions and properties of differential algebra that will be relevant in the next sections.

Given a field (or ring) $\FF$, a {\it derivation} is a map $\delta\colon \FF \to \FF$ that satisfies $\delta(f+g)=\delta(f) +\delta(g)$ and Leibniz' rule $\delta(f\cdot g)=f \delta(g) + \delta(f)g$ for all $f,g\in\FF$. A {\em differential field (or ring)} is a tuple~$(\FF,\delta).$ For a natural number $j$, we denote by $ f^{(j)} \coloneqq \delta^j(f)$ the $j$-th derivative of $f$, and by $f^{(0)}=f.$ In the multivariate case, one would define a differential polynomial ring with commuting derivations (see \Cref{sec:mDalgApp}), but we here restrict our presentation to the univariate case.
In this article, we will mainly consider the cases $\FF=\KK[x]$ or $\FF=\KK(x)$ together with the derivation $\delta =\partial \coloneqq \sfrac{\partial}{\partial x}.$

In differential algebra, the underlying object of study is the {\it differential polynomial ring} $\FF[y_1^{(\infty)},y_2^{(\infty)},\ldots,y_n^{(\infty)}]$, 
which corresponds to the set of polynomials in the indeterminates $y_i$ and their derivatives $y_i^{(j)},$ $i=1,\ldots,n$, $j\in \NN$. Below, we give a formal definition.

\begin{definition}
Let $(\FF,\delta)$ be a differential field or ring. The \emph{ring of differential polynomials in $y$ over $\FF$}, denoted $\FF[y^{(\infty)}],$ is the following differential ring.
It is the polynomial ring in infinitely many variables $y,y',y'',\ldots$
\begin{align}
\FF[y^{(\infty)}]\, \coloneqq \, \FF[y, y', y'', y^{(3)}, \ldots]
\end{align}
together with the derivation $\delta(y^{(j)}) \coloneqq y^{(j + 1)}$, extending the derivation from $\FF$. 
\end{definition}

In this setting, $y$ is called the {\em differential indeterminate}. The ring of differential polynomials in several differential indeterminates $y_1,\ldots, y_n$ is defined by iterating this construction. 

The {\em order} of a non-zero differential polynomial $p\in\KK [y^{(\infty)}]$ is the largest integer $n$ such that the coefficient of some monomial in $p$ containing $y^{(n)}$ is non-zero. 
\begin{definition}\label{def:diffideals}
An ideal $I \subset \FF[y^{(\infty)}]$ is called \emph{differential ideal} if $p \in I$ implies \mbox{$p'\in I$}.
\end{definition}
For $p_1,\ldots, p_k \in \FF[y^{(\infty)}]$, the ideal
\begin{align}
\langle p_1^{(\infty)},\ldots, p_k^{(\infty)}\rangle,
\end{align}
where $p_i^{(\infty)}$ denotes the set $\{ p_i^{(j)}\}_{j\in \NN}$,
is a differential ideal. Moreover, this is the smallest differential ideal containing $p_1,\ldots, p_k$, and we will denote it by $\langle p_1,\ldots, p_k\rangle^{(\infty)}$.

In the sequel, we will also need a truncated version of the differential polynomial ring. For $j\in \NN$, we denote by $\FF[y^{(\leq j)}]$ the differential ring
\begin{align}
\FF[y^{(\leq j)}] \,\coloneqq  \,\FF[y^{(\infty)}] / \langle y^{(j+1)}\rangle^{(\infty)} \,\cong \, \FF[ y, y', \ldots, y^{(j)}].
\end{align}
In particular, $\delta(y^{(j)})=0$ in $\FF[y^{(\leq j)}].$
For a differential ideal $I=\langle p_1,\ldots,p_n\rangle^{(\infty)}\subset \FF[y^{(\infty)}]$ and $j\in \NN$, we will denote by $I^{(\leq j)}$ the ideal 
\begin{align}
I^{(\leq j)} \, \coloneqq \, \langle p_1,p_1',\ldots,p_1^{(j)},\ldots,p_n,p_n',\ldots,p_n^{(j)}\rangle.
\end{align}

Also dynamical models fit well into that setting. Among others, they commonly arise in chemical reaction networks, see for instance \cite{ODEbase} for many examples. Recall that a {\em dynamical model} over $\FF$ (see \cite[Section 1.7]{Glebnotes}, \cite[Section 2.2]{pavlov2022realizing}) is a system of the form
\begin{align}\label{eq:dynsys}
\mathbf{y}'=\mathbf{A}(\mathbf{y},\mathbf{u}), \quad
\mathbf{z} = \mathbf{B}(\mathbf{y},\mathbf{u}),
\end{align}
where $\mathbf{y}=(y_1,\ldots,y_n)$, $\mathbf{z}~=~(z_1,\ldots,z_m)$, and $\mathbf{u}=(u_1,\ldots,u_l)$ are function variables referred to as the {\em state}, {\em output}, and {\em input} variables, respectively; $\mathbf{A} \in\FF[y_1,\ldots,y_n,u_1,\ldots,u_l]^n$, $\mathbf{B}~\in~ \FF[y_1,\ldots,y_n,u_1,\ldots,u_l]^m$. The {\em dimension} of system \eqref{eq:dynsys} is $n$. The system can be generalized to the case where $\mathbf{A}$ and $\mathbf{B}$ are vectors of rational functions.

We here are interested in dynamical models that relate to special D-algebraic functions. 
We consider systems $\mathcal{M}$ over $\FF=\KK(x)$ of the form
\begin{align}\label{eq:dynmodel}
\mathbf{y}'=\mathbf{A}(\mathbf{y}), \quad z=B(\mathbf{y}),
\end{align}
also called state-space system without input, where $\mathbf{A}=(A_1,\ldots,A_n)\in \KK(x)(y_1,\ldots,y_n)^n$ is a vector of rational functions,  
$B\in\KK(x)(y_1,\ldots,y_n)$, and $\mathbf{y}$ is the vector $(y_1, \ldots, y_n)$.
In order to put~\eqref{eq:dynmodel} in the context of differential algebra, let $Q$ be the common denominator of the system and write $A_i=a_i/Q$ for $1\leq i \leq n$, and $B=b/{Q}$, where $a_1,\ldots,a_n,b\in \KK(x)[y_1,\ldots,y_n]$. We consider the following $n+1$ differential polynomials:
\begin{align}\label{eq:poldynmodel}
 Q\,\mathbf{y}'-\mathbf{a}(\mathbf{y}), \ Q\,z - b(\mathbf{y}) \,\in \,\KK(x)[\mathbf{y}^{(\infty)},z^{(\infty)}].
\end{align}

Before recalling results from~\cite{Glebnotes} in Propositions~\ref{prop:prop7} and \ref{prop:minp}, slightly adapted, we recall that the {\it saturation} of an ideal $I$ in a ring $R$ by an element $s\in R$ is the ideal
\begin{align}
I\colon s^{\infty} \, \coloneqq \, \left\{r\in R \,|\, \exists \, n\in \NN: s^n\,r\in I\right\}.
\end{align}
Geometrically, if $R$ is a polynomial ring, the process of saturation removes those irreducible components from the algebraic variety defined by $I$  where $s$ vanishes. In our differential setting, this will avoid that the denominator of fractions of differential polynomials vanishes.

Now let ${\mathcal{M}}$ be a model as in~\Cref{eq:poldynmodel}.
\begin{proposition}\label{prop:prop7} 
Consider the differential ideal 
\begin{align}\label{eq:idprop7}
I_{\mathcal{M}}\,\coloneqq\, \langle Q\, \mathbf{y}'-\mathbf{a}(\mathbf{y}),\ Q\, z-b(\mathbf{y}) \rangle^{(\infty)} \colon Q^{\infty} \,\subset \,\KK(x)[\mathbf{y}^{(\infty)},z^{(\infty)}]\, .
\end{align}
\begin{enumerate}[(1)]
\item On the differential polynomial ring $\KK(x)[\mathbf{y}^{(\infty)},z^{(\infty)}]$, consider the lexicographic monomial ordering $\prec$ corresponding to any ordering on the variables such that
\begin{itemize}
\item[(a)] ${z}^{(j_1)} > {y_i}^{(j_2)}$ for all $j_1, j_2\in \NN$ and $i\in \{ 1, \ldots , n\}$,
\item[(b)] $z^{(j+1)}>z^{(j)}$ and $y_{i_1}^{(j+1)}>y_{i_2}^{(j)}$ for all $i_1, i_2,j\in\NN$.
\end{itemize}
Then the set of all the derivatives of \eqref{eq:poldynmodel} forms a Gr\"{o}bner basis of $I_{\mathcal{M}}$ w.r.t.\  $\prec$.
\item As a commutative algebra, $\KK(x)[\mathbf{y}^{(\infty)},z^{(\infty)}]/I_{\mathcal{M}}$ is isomorphic to $\KK(x)[y_1,\ldots,y_n]$. In particular, $I_{\mathcal{M}}$ is a prime differential ideal.
\end{enumerate}
\end{proposition}

The following proposition will be a key statement to prove that our algorithms terminate.
\begin{proposition}\label{prop:minp} 
Consider a non-zero polynomial $p$ in the elimination ideal $I_{\mathcal{M}}\cap\KK(x)[z^{(\infty)}]$ of lowest degree among the non-zero polynomials of the lowest order in~$z$. Define 
\begin{align}\label{eq:idealynsys}
I_{\mathcal{M},j}\, \coloneqq\, \langle (Q\,\mathbf{y}'-\mathbf{a}(\mathbf{y}))^{(<j)},\, (Q\,z-b(\mathbf{y}))^{(\leq j)}\rangle \colon Q^{\infty} \,\subset \, \KK(x)[\mathbf{y}^{(\leq j)},z^{(\leq j)}]\, .
\end{align}
Then $p\in I_{{\mathcal{M}},n}$. In particular, $p$ can be computed using elimination with Gr\"{o}bner bases.
\end{proposition}

\begin{remark}
Method II will build on the saturated ideal~\eqref{eq:idealynsys}, hence is valid only for generic solutions of the considered ADEs. Method I requires no such genericity assumption; it works for {\em all} solutions.
\end{remark}

The following example demonstrates how this theory will help to construct ADEs for rational expressions of D-algebraic functions.
\begin{example}\label{ex:dynmodel}
The Painlev\'{e} transcendent of type I, denoted $f(x)$ here, is a solution of
\begin{align}\label{ex:painleve}
 y''(x) \,=\, 6 \left(y(x)\right)^2 + x \, .
\end{align} 
We introduce new differential indeterminates $y_0$ symbolizing $f$ and $y_1$ symbolizing $f'$.  
Then, by \eqref{ex:painleve}, $y_1'=6y_0^2+x$. The problem of finding a differential equation fulfilled by $g=f^2$ is equivalent to finding a non-zero differential polynomial $p\in I\cap \KK[x][z^{(\infty)}]$, where $I$ denotes the differential ideal 
\begin{align}
I \,=\, \langle y_0' - y_1,\, y_1'- 6 y_0^2 - x, \, z-y_0^2\rangle^{(\infty)} \, \subset \, \KK[x][y_0^{(\infty)},y_1^{(\infty)},z^{(\infty)}]\,.
\end{align}
Our algorithms, that we present in later sections, will find that $g(x)$ is a zero of the differential polynomial
\begin{align}
p \,=\,  -16\, x^2 z^3 - 192\, xz^4 - 576\, z^5 + 4\, z^2 \left( z''\right)^2 - 4\, z \left({z'}\right)^2 {z''} + \left(z'\right)^4 \, \in \, \KK[x][z^{(\infty)}]
\end{align}
of order~$2$ and degree~$4$.
\end{example}

Note that the Maple 2022 command {\tt dsolve} gives no result for solving this equation; the solutions of \eqref{ex:painleve} cannot be expressed in terms of elementary functions and many well-known special functions. Despite that fact, our algorithms can compute differential equations for rational functions of solutions to~\eqref{ex:painleve}.

\section{D-algebraic functions and their closure properties}\label{sec:dalgclosure}
In this section, we address D-algebraic functions and their closure properties. These are functions which are solutions to ADEs. Just as an algebraic function is the zero of a polynomial, a D-algebraic function is a zero of a {\em differential} polynomial. We will seamlessly switch between differential polynomials and their associated ADEs.
Note that the term \mbox{``D-algebraic''} stands for ``differentially-algebraic'', where the letter D is not to be confused with the Weyl algebra~$D$, which is denoted by an italic letter.
If we speak of ``functions'', we always mean elements of some suitable $\KK(x)$-algebra~$\cF$ which is closed under differentiation, i.e., $\delta(\cF)\subseteq \cF$. In particular, this implies that $\cF$ is closed under addition, multiplication, and multiplication by rational functions in~$x$. Often, one requires $\mathcal{F}$ be a domain or a field. The derivation should moreover be compatible with that of~$\KK(x),$ i.e.,
\begin{align}
    \delta \left(r\cdot f\right) \ =\  \frac{\partial r}{\partial x}\cdot f \,+\, r\cdot \delta\left( f \right)
\end{align}
for all $ r\in \KK(x)$ and $f\in \mathcal{F}$.
We will not need to be more restrictive on that. In particular, $\mathcal{F}$ does not have to contain $\KK(x),$ but typically does so in applications.

\subsection{Definitions}
In order to emphasize the variables on which the function $f$ depends on, we will write $f(x)$ for a function $f$ of a variable~$x$, and so on.
For a differential polynomial $p\in \FF[y^{(\infty)}]$ and a function $f(x)$, we denote by $p(f)$ the evaluation of $p$ at $y=f$. For instance, the differential polynomial $p=y'-y\in\KK[x][y^{(\infty)}]$ vanishes at $y=\exp$, since $\exp'(x)-\exp(x)=0$. 

\begin{definition} \label{def:dalgfun}
A function $f(x)$ is {\em D-algebraic (over $\KK(x)$)} if there exists a differential polynomial $p\in \KK[x][y^{(\infty)}]\setminus \{0\}$ that vanishes at $y=f$, i.e., $p(f)=0$.
\end{definition}
\begin{remark}
Every D-algebraic function also satisfies an ADE with {\em constant} coefficients. Deriving an ADE with constant coefficients from a given ADE with polynomial coefficients requires  an elimination step, which increases the order of the ADE. This is undesirable for our further undertakings and we therefore do not restrict ourselves to constant~coefficients.
\end{remark}

We will denote the class of univariate D-algebraic functions by $\cA$. Without further mentioning, we always assume that an appropriate algebra $\mathcal{F}$ of functions had been fixed. Since a differential polynomial involves only finitely many coefficients, $\KK$ can be chosen to be a field extension of~$\QQ$ of finite transcendence degree and we will tacitly assume so in computations.

\begin{definition}
Let $f(x)$ be a D-algebraic function. The {\em order of $f$} is the minimal $n\in \NN$ such that there exists a non-zero differential polynomial $p$ of order $n$ for which $p(f)=0$.
\end{definition}

\noindent Before we carry out operations with D-algebraic functions, we revisit their closure~properties.

\subsection{Closure properties of D-algebraic functions}
We recall closure properties of D-algebraic functions together with their proofs. Similar expositions can be found in \cite[Section 6]{raschel2020counting} and \cite[Section 3]{vdH19}. Clearly, one can characterize D-algebraic functions as follows.
\begin{lemma}\label{prop:finitetr} 
A function $f(x)$ is D-algebraic if and only if the field $\KK(x,f^{(\infty)})$ has finite transcendence degree over $\KK(x)$, where $f^{(\infty)}$ denotes the set $\{\partial^k(f)\}_{k\in \NN}.$
\end{lemma}

\begin{definition}
The {\em transcendence degree} of a D-algebraic function $f(x)$ is 
\begin{align}
\tr(f) \, \coloneqq \, \tr(\KK(x,f^{(\infty)}),\KK(x)) \, .
 \end{align}
\end{definition}
The transcendence degree will play a key role in the proof of the following proposition.

\begin{proposition}\label{prop:fieldprop} 
The set of D-algebraic functions is a field.
\end{proposition}
\begin{proof} It is enough to show that the class of D-algebraic functions is closed under addition, multiplication, and taking reciprocals.
Let $f$ and $g$ be D-algebraic functions in $x$ over $\KK$. By \Cref{prop:finitetr}, $\tr(f)<\infty$ and $\tr(g)<\infty$. Since $\KK(x,(f+g)^{(\infty)})\subset \KK(x,f^{(\infty)},g^{(\infty)})$, we have that 
\begin{align*}
\tr(f+g)  &\,\leq\, \tr(\KK(x,f^{(\infty)},g^{(\infty)}),\KK(x))\\
&\,=\, \tr(\KK(x,f^{(\infty)}),\KK(x)) + \tr(\KK(x,g^{(\infty)}),\KK(x)) \,<\, \infty \, .
\end{align*}
We conclude by \Cref{prop:finitetr}. The proof for multiplication is the same as for addition since $\KK(x,(fg)^{(\infty)})\subset \KK(x,f^{(\infty)},g^{(\infty)})$. Since $\tr(f)=\tr(1/f)$, also $1/f$ is D-algebraic. To construct an ADE for $1/f$, one proceeds as follows. Given $p$ of order~$n$ with $p(f)=0$, one finds a differential polynomial with $g=1/f$ as zero by substituting $f$ by~$1/g$ in $p(f)$ and taking the numerator of the resulting rational expression in $g,\ldots,g^{(n)}$.
\end{proof}

In other words, every rational expression of D-algebraic functions is D-algebraic. We will now argue that the field of D-algebraic functions is also closed under composition.
\begin{proposition}\label{prop:compDalg} 
Let $f(x)$ and $g(x)$ be two D-algebraic functions. Then, whenever it is defined, also their composition $f\circ g$ is D-algebraic.
\end{proposition}
\begin{proof}
By the chain rule and Leibniz' rule,
\begin{align}\label{eq:chainrule}
\begin{split}
(f(g))' & \,\,=\,\, g' \, f'(g) \, ,\\
(f(g))''&\,\,=\,\, g''\, f'(g) \,+\, \left(g'\right)^2 f''(g) \, ,\\
(f(g))^{(3)} &\,\,=\,\, g^{(3)}\,f'(g) \,+\, g''\,g'\,f''(g) \,+\, \left(g'\right)^3 f^{(3)}(g)\, ,
\end{split}
\end{align}
and so on.
Hence all derivatives of $f(g)$ can be expressed as (linear) polynomial expressions in $f'(g), f''(g),f^{(3)}(g),\ldots$ with coefficients in $\KK[g', g'', g^{(3)},\ldots]$. Therefore,
\begin{align*}
\left(f(g)\right)^{(j)}\in\KK[g,g',\ldots][f(g),f'(g),\ldots]
\end{align*}
for all $j\in \NN$. Since $f$ and $g$ are D-algebraic, we deduce that
\begin{align*}
\tr(f(g)) \, \leq \, \tr(f) + \tr(g) \,<\, \infty
\end{align*}
and thus $f(g)$ is D-algebraic.
\end{proof}

By Propositions~\ref{prop:fieldprop} and~\ref{prop:compDalg}, if $f$ is D-algebraic, then also $\sqrt[k]{f}$ for all $k\in\NN$, $\exp(f)$, $\log(f)$, $\cos(f)$, etc., and rational expressions of them are D-algebraic.

\begin{example} 
Since $f(x)= \sqrt{x}$ is algebraic, fulfilling the polynomial equation $(f(x))^2=x$, considered as an ODE of order $0$, the function $h(x)\coloneqq f(g(x))=\sqrt{g(x)}$, where $g$ is the Painlev\'{e} transcendent of type~I (see~\eqref{ex:painleve}), is D-algebraic. 
It fulfills
\begin{align}
-x - 6 \left(h(x)\right)^4 + 2 \left(h'(x)\right)^2 + 2\, {h''(x)}\, {h(x)} \,=\, 0 \, .
\end{align}
The next subsections will explain in which sense this ADE for $h$ is  minimal with respect to the input ADEs for $f$ and $g$.
\end{example}

Also inverse functions of D-algebraic functions are D-algebraic.
\begin{proposition}\label{prop:invDalg}
Let $X$ and $Y$ be two sets and $f\colon X \longrightarrow Y$ an invertible D-algebraic function. Then also its inverse function $g\coloneqq f^{-1}$ is D-algebraic.
\end{proposition}
\noindent We here present a constructive proof building on the proof of \Cref{prop:compDalg}.
\begin{proof}
Since $f$ is D-algebraic, there exists a polynomial $p$ and $n\in \NN$ such that 
\begin{equation}\label{eq:invproofeqP}
    p(x,f(x),f'(x),\ldots,f^{(n)}(x))\, = \, 0 \, \quad \text{for all } x\in X.
\end{equation}
Since $g$ is the inverse function of $f$, we can rewrite \eqref{eq:invproofeqP} as 
\begin{align}\label{eq:invproofeqP2}
\begin{split}
    0 &\,=\, p(g(y),f(g(y)),f'(g(y)),\ldots,f^{(n)}(g(y))) \quad \text{for all } y \in Y. 
\end{split}
\end{align}
Using \eqref{eq:chainrule} with the derivatives w.r.t.\  $y$,
we deduce that for all $j\leq n$
\begin{align}
    f^{(j)}(g(y)) \,\in \,\KK(y,g(y),g'(y),\ldots, g^{(n)}(y)) \quad \text{for all } y\in Y.
\end{align} 
After substituting into \eqref{eq:invproofeqP2} and taking the numerator of the resulting expression, we obtain a differential equation that~$g$ fulfills.
\end{proof}
\begin{example}
Consider $f(x)=\exp(x)$. It is a zero of the differential polynomial\linebreak $p=w'-w$. Starting from $p$, we are going to construct an ODE fulfilled by its inverse function $g(y)=\ln(y)$. Then $p(f(x))=p(f(g(y))) = f'(g(y)) - y$. Now $(f(g(y)))'=y'=1$.  After substitution into $p$, we obtain the rational expression  $(1-y g'(y))/g'(y)$. Indeed, $g=\ln$ is a zero of the expression in the numerator.
\end{example}

More sophisticated examples for ADEs for inverse functions can be computed using our implementation. We conclude this subsection by giving two more closure properties of~$\cA$.
\begin{lemma}\label{lem:antiderder}
Let $f\in \cA$. Then also all antiderivatives $g$ of $f$, i.e., $g\in \cA$ for which $g'=f$, in case of existence, and $f'$ are D-algebraic.
\end{lemma}
\begin{proof}
Let $f\in \cA$ be D-algebraic and $p\in \KK[x][y^{(\infty)}]$ such that $p(f)=0$. We can suppose that\linebreak $p\in \KK[x][y^{(\infty)}]\setminus \KK[x][y',y'',\ldots]$ and moreover that $p$ is irreducible so that $p$ and $p'$ are coprime polynomials. The resultant $q \colon= \res_y(p, p') \in \KK[x][y',y'',\ldots]$ is of the form $Ap+Bp'$ for some $A,B\in \KK[x][y^{(\infty)}]$. Since $p$ and $p'$ are coprime, their resultant is non-zero. The equation for the resultant encodes an ADE which is fulfilled by~$f'$. To find an ADE for antiderivatives of~$f$, one only needs to replace $y^{(j)}$ by $y^{(j+1)}$ in $p$.
\end{proof}

\subsection{Multivariate D-algebraic functions}\label{sec:multiDalg}
In this subsection, we briefly address the multivariate case. In the meantime, a more detailed discussion is contained in Section~$3$ of the follow-up work~\cite{teguia2023arithmetic}. The definition of  D-algebraic functions generalizes to multivariate functions $f(x_1,\ldots,x_k)$. For that, we work over the differential ring $\KK[x_1,\ldots,x_k]$ with the $k$ pairwise commuting derivations $\partial_i=\sfrac{\partial}{\partial x_i}$, \linebreak $i=1,\ldots, k,$ and with the differential polynomial ring
\begin{align}\label{eq:multi}
\KK[x_1,\ldots,x_k][y^{(\infty,\ldots,\infty)}] \, ,
\end{align}
where $y^{(\infty,\ldots,\infty)}$ denotes the set $\{y^{(i_1,\ldots,i_k)} \}_{i_1,\ldots,i_k \in \NN}$ with $y^{(i_1,\ldots,i_k)}=\partial_1^{i_1}\cdots \partial_k^{i_k}(y).$

\begin{definition}
A multivariate function $f(x_1,\ldots,x_k)$ is {\em D-algebraic (over $\KK(x_1,\ldots,x_n)$)} if there exists a differential polynomial $p\in \KK[x_1,\ldots,x_k][y^{(\infty,\ldots,\infty)}]$ such that $p(f)=0.$ 
\end{definition} 
The associated differential equations are {\em partial} algebraic differential equations.

\begin{proposition}
The set of multivariate D-algebraic functions is a field.
\end{proposition}
\begin{proof} 
Applying a similar line of argumentation as in the proof of \Cref{prop:fieldprop} to the fields $\KK(x_1,\ldots,x_k,f^{(\infty,\ldots,\infty)})$, and so on,  proves the claim.
\end{proof}

\begin{example} The functions $f(x_1,x_2)\coloneqq 1/(1-\exp(x_1-x_2))$ and $g(x_1,x_2)\coloneqq x_2/(x_1+x_2)$ are both D-algebraic as they satisfy the differential equations 
\begin{align}\label{ex:multvar1}
& y^{(1,0)}(x_1,x_2) + y^{(0,1)}(x_1,x_2) \,=\, 0
\end{align}
and
\begin{align}\label{ex:multvar2}
(x_1+x_2)\,y_2(x_1,x_2) - x_2 \,=\, 0 \, ,
\end{align}
respectively. A generalization of our second method finds that the differential polynomial
\begin{equation}
(x_1+x_2)^2\,y^{(1,0)}(x_1,x_2) + (x_1+x_2)^2\,y^{(0,1)}(x_1,x_2)-x_1+x_2 \label{ex:multivar3}
\end{equation}
vanishes at sums of solutions to~\eqref{ex:multvar1} and~\eqref{ex:multvar2}; in particular, $f+g$ is a zero of~\eqref{ex:multivar3}.
\end{example}
For the rest of the article, we return to the univariate case.

\section{Arithmetic with D-algebraic functions}\label{sec:dalgfunc}
We will now explain how to compute ADEs resulting from operations with differential polynomials. We will describe two methods. The first method is based on jets and builds on iterative elimination with Gr\"{o}bner bases. Our second method computes differential equations whose order is within a fixed range. 

\subsection{Method I}\label{sec:directmth}
Our first method builds on jets. It does not make assumptions a priori on the order of the sought differential equation. We refer our readers who are unfamiliar with these concepts to the appendix, where we recall some basics about algebraic varieties and jet schemes.

For computations, we use the Macaulay2 package {\tt Jets} \cite{JetsSource}.   
Below, we give a detailed explanation on how Method I works. Let $\alpha\in\{+, -, \times, /\} $ denote an arithmetic operation
\begin{align}
\alpha \colon \, \cA\times \cA \longrightarrow \cA, \qquad (f,g)\mapsto \alpha(f,g) \,\coloneqq \,f \alpha g
\end{align}
on the class $\cA$ of D-algebraic functions.
Let $f,g\in\cA$ and $h=\alpha(f,g)$.
We would like to find an ADE for $h$ constructed from ADEs $p$ for $f$ and $q$ for $g$. Let $p  \in \KK[x][y^{(\infty)}]$, $q \in \KK[x][z^{(\infty)}]$ be differential polynomials of order $n$ and $m$, respectively, such that $p(f)=0$ and $q(g)=0$. 

We introduce a new differential indeterminate $w$ and denote by $R_{\alpha}\in\KK[x][y^{(\infty)},z^{(\infty)},w^{(\infty)}]$ the numerator of $w-\alpha(y,z) \in \KK[x](y^{(\infty)}, z^{(\infty)}, w^{(\infty)})$. It encodes a relation among the differential indeterminates $y,z,w$ and their derivatives. For $\alpha = /$, the numerator of\linebreak $w-y/z=(wz-y)/z$ is the differential polynomial $R_{/}=wz-y\in \KK[x][y^{(\infty)},z^{(\infty)},w^{(\infty)}].$ For $\alpha \in \{+,-,\times\}$, $R_{\alpha}$ is simply $w-\alpha(y,z)$. We consider the differential ideal 
\begin{align}\label{eq:idI}
I_{\alpha} \,\coloneqq \,\langle p,\,q,\,R_{\alpha}\rangle^{(\infty)} \,\subset \, \KK[x][y^{(\infty)},z^{(\infty)},w^{(\infty)}] \, .
\end{align} 
For $j\in \NN$, define $I_{\alpha}^{(\leq j)}$ as follows:
\begin{align}\label{eq:idealIj}
\begin{split}
I_{\alpha}^{(\leq j)} \,= \, \langle p,p',\ldots,p^{(j)},q,q',\ldots,q^{(j)},R_{\alpha},R_{\alpha}',\ldots,R_{\alpha}^{(j)} \rangle \, \subset\, \KK[x][y^{(\leq n+j)}, z^{(\leq m+j)}, w^{(\leq j)}].
\end{split}
\end{align} 
Clearly, this yields the ascending chain of ideals $I_{\alpha}^{(\leq 0)} \subseteq I_{\alpha}^{(\leq 1)}  \subseteq \cdots \subseteq I_{\alpha}$. 

The isomorphism~\eqref{eq:isojets} induces the isomorphism 
\begin{align}\label{eq:isotrunc]}
\begin{split}
&\KK[x]\{y_{\ell \leq n + j}, z_{\ell \leq m + j}, w_{\ell \leq j} \}] / \Tilde{I}_j \ \stackrel{\cong}{\longrightarrow} \ \KK[x] [y^{(\leq n+j)}, z^{(\leq m+j)}, w^{(\leq j)}]/ I_{\alpha}^{(\leq j)} \, ,\\& \qquad \qquad \quad x\mapsto x , \quad y_{\ell} \mapsto \frac{y^{(\ell)}}{\ell!}, \quad  z_{\ell} \mapsto \frac{z^{(\ell)}}{\ell!}, \quad w_{\ell} \mapsto \frac{w^{(\ell)}}{\ell!}
\end{split}
\end{align}
for the truncated ideal. The ideal $\Tilde{I}_j$ is obtained as follows. We first consider the power series $p(y(t)),$ $ q(z(t)),$ $ R_\alpha(w(t))$ in $t$, arising from the $j$-jets of $V(p,q,R_{\alpha})$ as explained in the appendix, where we moreover exploit the differential dependencies of the variables $y,y',y''$, and so on.  Whenever $x$ occurs in one of the power series $p(y(t)), q(z(t)), R_\alpha(w(t))$, we replace $x$ by $x+t$. The ideal $\Tilde{I}_j$ then is the ideal generated by the coefficients of $1,t,\ldots,t^j$ in the resulting series. To construct the ideal $I_\alpha^{(\leq j)}$ in Macaulay2, we compute the ideal {\tt jets(j,ideal(p,q,R_alpha))} and adapt the coefficients of the defining polynomials using the isomorphism~\eqref{eq:isojets}. In Macaulay2, for the autonomous case (i.e., \mbox{$p\in \KK[y^{(\infty)}]$}, \mbox{$q\in \KK[z^{(\infty)}]$}), one starts from the ring whose variables are called 
{\tt y\_0..y\_ord(p)}, {\tt z\_0..z\_ord(q)}, and {\tt w\_0}. All the computations in Macaulay2 are done over the base field~$\QQ$. In the case of non-autonomous differential polynomials, i.e., those containing also the independent variable~$x$, we encode $x$ in Macaulay2 with the following trick. We introduce one more indeterminate ``$x_0 = x$'' and add the element $x_1 - 1$ to the ideal~$I$. 
\begin{proposition}\label{prop:directmeth}
There exists $N\in \NN$ such that
\begin{align}\label{eq:elim}
I_{\alpha}^{(\leq N)} \cap \KK[x][w^{(\infty)}] \,\neq \, \langle 0 \rangle \, .
\end{align}
\end{proposition}
\begin{proof} 
Let $P_1, \ldots, P_k $ be the prime components of $I_\alpha$, i.e, $\sqrt{I_\alpha} = \cap_i P_i$. For every\linebreak $1 \leq i \leq k$, $p, q \in P_i$, which means that the images of $y$ and $z$ in $\KK[x][y^{(\infty)}, z^{(\infty)}, w^{(\infty)}]/P_i$ are differentially-algebraically dependent over $\KK[x]$. Hence the image of $\alpha(y,z)$ is D-algebraic which implies that there exists a non-zero differential polynomial $Q_i \in \KK[x][w^{(\infty)}]$ which belongs to $P_i$.
Then $\Tilde{Q}\coloneqq Q_1\cdots Q_k$ is contained in $ \sqrt{I_\alpha}$ and hence a power of $\Tilde{Q}$ is contained in $I_\alpha$. Therefore, there exists a non-zero differential polynomial $Q$ which depends only on $w$ and belongs to~$I_\alpha$. Let $\ell=\ord(Q)$ and consider the ascending chain of ideals
\begin{align}
I_{\alpha}^{(\leq 0)} \cap \KK[x][w^{(\leq \ell)}] \,\subseteq \, I_{\alpha}^{(\leq 1)} \cap \KK[x][w^{(\leq \ell)}] \, \subseteq \,  \cdots \, .
\end{align}
By Noetherianity of $\KK[x][w^{(\leq \ell)}]$, this chain stabilizes and hence there exists $N$ such that 
$Q$ is contained in $I_{\alpha}^{(\leq N)} \cap \KK[x][w^{(\leq \ell)}].$
\end{proof}
Note that elements of $I_{\alpha}^{(\leq N)}\cap \KK[x][w^{(\infty)}]$ have order at most $N$. Moreover, there might be several differential polynomials of minimal degree among those of minimal order.
To arrive at a non-trivial elimination ideal, taking $N$ to be $\ord(p)+\ord(q)$ is in general not sufficient, as the following example demonstrates.
\begin{example}
For $p=y\,y''-(y')^2$ and $q=z^2+(z')^4$, the elimination ideal $I_N$  from~\eqref{eq:elim} for $\alpha=+$ is trivial for $N=\ord(p)+\ord(q)=3$. Yet, there might be a non-zero differential polynomial of order~$3$ in $I_k$ for some $k\geq 4$.
\end{example}

Thus, in order to compute a differential equation for $h=\alpha(f,g)$, the first method iterates the elimination computation 
\begin{align}\label{eq:intersectId}
I_{\alpha}^{(\leq j)} \cap \KK[x][w^{(\leq j)}] \, ,  \qquad j=1,2,3,\ldots
\end{align}
until the intersection is non-trivial. From the output, namely the Gr\"{o}bner basis of the intersection ideal, we pick a differential polynomial of lowest degree among those of lowest order.
The general strategy consists in defining a differential polynomial $R_{\alpha}$ that encodes the operation $\alpha$ that we wish to perform. This can be generalized to arithmetic operations with several D-algebraic functions. 

\smallskip

We now present the construction of an algebraic differential equation for  $\alpha=\circ,$ i.e., for the composition of D-algebraic functions. Although one can still define $R_{\circ}$ in this case, we here give a more suitable approach for algebraic manipulations.
Let $p \in \KK[x][y^{(\infty)}]$ such that $p(f)=0$. Then for every function $g$ we also have $p(f(g)) =0$. Let $h = f \circ g$. Then 
\begin{align}
h^{(j)} \in \KK[x][g^{(\infty)}, f(g), f'(g), f''(g),\ldots ]
\end{align}
and it is homogeneous of degree $1$ in the $f^{(i)}(g)$.
If $ h^{(j)}(x) = \sum_{i=1}^j q_i(g(x)) f^{(i)}(g(x))$, then
\begin{align}
h^{(j+1)}(x) \,=\, \sum_{i =1}^j \left(q_i(g(x))\right)' f^{(i)}(g(x)) + q_i(g(x)) g'(x) f^{(i +1)}(g(x)) 
\end{align}
by Leibniz' rule.
This recurrence relation will be used in the construction of the ideal containing the ADE for the composition.
Let now $p \in \KK[x][y^{(\infty)}]$ be of order $m$ and \mbox{$q \in \KK[x][z^{(\infty)}]$} of order $n$ and let $I = \langle p , q\rangle^{(\infty)} \in \KK[x][y^{(\infty)}, z^{(\infty)}]$.
In the case that $p$ involves non-constant coefficients, i.e., coefficients involving~$x$, one moreover substitutes these~$x$ by~$z$ in~$p$ and its derivatives.
For every $j\in \NN$, we define the differential polynomial $S_j\in \KK[x][y^{(\infty)},z^{(\infty)}]$ recursively as follows:
\begin{align}\label{eq:compositionEqs}
\begin{split}
\begin{cases}
S_0 = y, \\
\text{if } S_j = \sum_{i =0}^j q_i(z) y^{(i)},  \text{ then } S_{j+1} = \sum_{i =0}^j (q_i(z))'y^{(i)} + q_i(z) z' y^{(i +1)}.
\end{cases}
\end{split}
\end{align}
We consider the (not necessarily differential) ideal
\begin{align}\label{eq:H}
H  \,\coloneqq \, I + \langle \{ w^{(i)} - S_i (y,z) \,|\, i\in \NN \} \rangle \,\subset \,\KK[x] [y^{(\infty)}, z^{(\infty)} , w^{(\infty)}] \, .
\end{align}
For $j \in \NN$ we define 
\begin{align}\label{eq:Hj}
H_j \, \coloneqq \, I^{(\leq j)} + \langle \{ w^{(i)} - S_i (y,z)\, | \, i \leq j + \min(m,n)\} \rangle  \,\subset \, \KK[x] [y^{(\infty)}, z^{(\infty)} , w^{(\infty)}]\, .
\end{align}
\begin{proposition}
There exists $N \in \NN$ such that
\begin{align}\label{eq:HN}
 H_N \cap \KK[x][w^{(\infty)}] \, \neq \,\langle 0 \rangle \, .
\end{align}
\end{proposition}
\begin{proof}
Let $P_1, \ldots, P_k $ be the prime components of $I$, i.e., $\sqrt{I} = \cap_i P_i$.  Suppose that $\ord(p) = n$ and $\ord(q) = m$. Then for every $1 \leq i \leq k$, $\tr(\KK[x][y^{(\infty)}, z^{(\infty)}]/P_i , \KK(x)) \leq m + n$. Therefore, the image of the family $\{ S_0, \ldots , S_{m + n} \}$  in $\KK[x][y^{(\infty)}, z^{(\infty)}]/P_i$ is algebraic, hence it is algebraic in $\KK[x][y^{(\infty)}, z^{(\infty)}]/\sqrt{I}$. Therefore, there exists $Q \in \KK[x][w^{(\infty)}] $ such that $Q \in H$. Let $\ell := \ord(Q)$. Using Noetherianity of $\KK[x][w^{(\leq)}]$ as in \Cref{prop:directmeth}, we conclude that there exists $N$ such that the elimination ideal
$H_N \cap \KK[x][w^{(\infty)}] \neq\langle 0 \rangle$
is non-trivial.
\end{proof}

We give some examples to illustrate how to algorithmically determine the differential equation  for the composition of D-algebraic functions.
\begin{example}
Let $p=y'-y$ and $q = z^2 + 2 z'$.
Taking $N = 1$ we have  $S_0 = y ,$ $S_1 = z' y',$ $S_2 = z'' y' - (z')^2 y'' $. We eliminate the variables $y, y' , y'', z , z' , z''$ from the following system of polynomial equations:
\begin{align}
\begin{split}
\begin{cases}
    w - y \,=\, 0, \\
    w' - z' y' \,=\, 0,\\
    w'' - z'' y' - \left(z'\right)^2 y'' \,=\, 0, \\
    y \,=\, y' \,=\, y'',\\
    z^2 + 2z' \,=\, 0, \\
    2zz' + 2 z'' \,=\, 0.
\end{cases}
\end{split}
\end{align}
We obtain the differential polynomial $ (w')^4 - 2 w (w')^2 w'' + w^2 (w'')^2 + 2 w (w')^3 \in H_1$. For any constant~$b$, the function $f_b(x)\coloneqq \exp(2/(2b+x))$ is a zero of this differential polynomial. These functions are holonomic. In contrast to the given ADE, which is independent of $b$, linear differential operators that annihilate $f_b$ depend on~$b$.
\end{example}
\begin{example}
As a second example, we consider $p=y''+y$ and $q = z' - x z$. For $N = 2$, we eliminate $y,y' , y'', y^{(3)}, z,z', z'', z^{(3)}$ from the system
\begin{align}
\begin{split}
\begin{cases}
    w - y \,=\, 0 ,\\
    w' - z' y' =0,\\
    w'' - z'' y' - \left({z'}\right)^2 y'' \,=\, 0, \\
    w^{(3)} - z^{(3)}y' - 3 z' z'' y'' - \left({z'}\right)^3 y^{(3)} \,=\, 0,\\
    y  + y'' \,=\, 0,\\
    y' + y^{(3)} \,=\, 0,\\
    z' \,=\, x z, \\
    z'' \,=\, z + x z', \\
    z^{(3)} \,=\, 2z' + xz'' .
\end{cases}
\end{split}
\end{align}
We get
$(2 x^4+ 3 x^2 +3) w w' + (x^3+x) ({w'})^2 - 3( x^3+x) w w''  - x^2 w' w'' + x^2 w w^{(3)} \in H_2.$
\end{example}

\subsection{Method II}\label{sec:secmeth}
Our second method defines bounds for the order of the sought differential equations. Let $p\in \FF[y^{(\infty)}]$ and $q\in \FF[z^{(\infty)}]$ be of order $n$ and $m$, respectively. For simplicity, we introduce the following notation.

\begin{definition}
A differential polynomial $p\in\FF[y^{(\infty)}]$ of order $n$ is called {\em linear in its highest order term (l.h.o.)} if, in the expression of $p$, $y^{(n)}$ appears only linearly.
\end{definition}
If $p$ is l.h.o.\  and of order $n,$  $y^{(n)}$ can be written as a rational function $r_p$ of $x,y,\ldots,y^{(n-1)}$. From a differential polynomial which is of higher degree in its highest order term, one arrives at an l.h.o.\  polynomial by differentiating it once, and we will always do so in our applications.

\begin{example}
The differential polynomials $y''-(y')^2+xy+x$ and ${y''}({y'})^2y+2y^2$ are l.h.o., whereas $q=(y')^3+xy+1$ is not l.h.o. But its derivative $q'=3y''(y')^2+y+xy'$ is l.h.o.
\end{example}

When studying zeros of $p,q$, we will  have to disregard some of them in the non-l.h.o.\ case. The following example demonstrates the necessity of doing so.

\begin{example} 
Let $p=(y^{(n+1)})^{d+1}+(y^{(n)})^{d+1}+c$, where $d, n\in \NN,$ $d\geq 1$, and $c\in\KK$. Since $p$ is not l.h.o., we derive $p$ once and get
\begin{align}
(d+1)\, y^{(n+1)}\, \left(y^{(n+2)}\,\left(y^{(n+1)}\right)^{d-1}+\left(y^{(n)}\right)^d\right).
\end{align}
The zeros of $p$ whose $(n+2)$-nd derivative can {\em not} be written as a rational function of $x,y,\ldots,y^{(n+1)}$ are those $f\in \cF$ for which $f^{(n+1)}=0$; e.g., polynomials of degree at most~$n$. 
\end{example}

Thus, in particular, we neglect the special case of polynomial solutions of low degree. For them, computations can be done separately.
Since we work with coefficients in $\KK[x]$---and computationally over $\KK(x)$---an arithmetic operation between a non-polynomial D-algebraic function $f$ and a polynomial $g\in \KK [x]$ can be seen as a unary operation. Using Method~I, one only needs the differential polynomial $p$ with $p(f)=0$ as specified in~\eqref{eq:idI}; $R_{\alpha}$ would encode the rational relation between $g$ and the zeros of $p$, namely $R_{\alpha}=w-\alpha(y,g)$, and~$I_{\alpha}=\langle p, R_{\alpha}\rangle$. 

One may also assume that $p$ and $q$ are irreducible as multivariate polynomials.

We now investigate the dimension of the dynamical model $\mathcal{M}_{\alpha(p,q)}$ of the form~\eqref{eq:dynmodel} derived from $p$ and $q$, such that the second equation in \eqref{eq:dynmodel} encodes the operation $\alpha$  we wish to perform. The differential polynomials we are seeking are the elements of the associated differential ideal $I_{\mathcal{M}_{\alpha},j}$ as in \Cref{prop:minp} for some $j$. The dimension of $\mathcal{M}_{\alpha(p,q)}$ is an upper bound for the order of differential polynomials of minimal order among those that can be derived from $\mathcal{M}_{\alpha(p,q)}$. We will denote this number by $\mathfrak{o}_{\alpha}(p,q)$.

\begin{proposition}\label{thm:boundldf}
Let $\alpha\in\{+, -, \times, /\}$ and $p\in \FF[y^{(\infty)}]$, $q\in \FF[z^{(\infty)}]$ of order $n$ and $m$, respectively. 
If $p$ and $q$ are l.h.o., then $\mathfrak{o}_{\alpha}(p,q) \leq m+n$.
\end{proposition} 
\begin{proof} Since $p$ and $q$ are l.h.o., we can write $y^{(n)}$ and $z^{(m)}$ as rational functions 
\begin{align}
y^{(n)} \,=\, r_p \left( x,y,\ldots,y^{(n-1)}\right) \quad \text{and}\quad z^{(m)}\,=\, r_q\left(x,z,\ldots,z^{(m-1)}\right).
\end{align}
Define $m+n$ new differential indeterminates
\begin{align}\begin{split}
u_1 \,=\, y, \quad u_2\,=\,y',\   \ldots \  , \quad  u_{n} \,=\, y^{(n-1)},\\
u_{n+1} \,=\,z, \quad u_{n+2}\,=\,z', \ \ldots \ ,  \quad u_{n+m} \,=\, z^{(m-1)}.
\end{split}
\end{align}
This gives rise to the dynamical model
\begin{align}\label{eq:dynsysalpha}
\begin{split} 
u_1' \,=\, u_2, \, \   \ldots \  ,  \ u_{n-1}'\,=\,u_n,  \quad u_{n}' \,=\, r_p(x,u_1, \ \ldots, \ u_{n}), \\
u_{n+1}' \,=\, u_{n+2}, \ \ldots \ , \ u_{n+m-1}' \,=\, u_{n+m}, \quad   u_{n+m}'=r_q(x,u_{n+1},\ldots,u_{n+m}),\\
w \,=\, \alpha(u_1,u_{n+1}),
\end{split}
 \end{align} 
 which we denote by $\mathcal{M}_{\alpha(p,q)}$. System~\eqref{eq:dynsysalpha} is of the same type as~\eqref{eq:dynmodel} over the differential field $\KK(x)$. Therefore, by \Cref{prop:minp}, we deduce that $\mathfrak{o}_{\alpha}(p,q)\leq n+m$.
\end{proof}

\begin{corollary}\label{cor:boudldf}
Let $\alpha\in\{+, -, \times, /\}$ and $p\in \FF[y^{(\infty)}]$, $q\in \FF[z^{(\infty)}]$ be of order $n$ and $m$, respectively. Then  $\mathfrak{o}_{\alpha}(p,q) \leq m+n+2$.
\end{corollary}
\begin{proof}
If $p$ and $q$ are linear in their highest order terms, then \Cref{thm:boundldf} applies. Suppose that at least one of $p$ and $q$ is not l.h.o. Then the differential polynomial
$\delta(p) = y^{(n+1)} p_1 + p_2$ is of order $n+1$ and linear in its highest order term. If $q$ is of higher degree in its highest order term, we proceed similarly. We conclude by \Cref{thm:boundldf} that $\mathfrak{o}_{\alpha}(p,q)\leq m+n+2$.
\end{proof}

\begin{remark}[Unary arithmetic with Method II]
Let $f\in \cA$ and $p$ be a differential polynomial that vanishes at $f$. In order exclude polynomials of degree at most $\ord(p)$, we assume that $f^{(\ord(p)+1)}\neq 0$. For each element in $\KK(x,f)$, Method~II requires only one input differential polynomial, and yields a differential polynomial of order at most $\ord(p)+1$. In particular, rational  functions in $f$ and polynomials in~$x$ can be treated via unary~operations.  
\end{remark}

Our second method is based on \Cref{cor:boudldf}; the underlying algorithm finds a differential equation of order at most $m+n$, $m+n+1$, or $m+n+2$ satisfied by $\alpha(f,g)$, depending on whether $p$ and $q$ are linear in their highest order terms or not. 

We now investigate the order $\mathfrak{o}_{\circ}$ of the resulting differential polynomials for the operation~$\circ$ of composing two D-algebraic functions.

\begin{proposition}\label{thm:cboundldf} 
Let $p\in \FF[y^{(\infty)}]$ and $q\in \FF[z^{(\infty)}]$ of order $n$ and $m$, respectively.\linebreak If $p$~and~$q$ are l.h.o., then $\mathfrak{o}_{\circ}(p,q) \leq m+n .$
\end{proposition}
\begin{proof}
In the previous section, we saw how to obtain a rational expression (free of $x$), say $R\in \KK(z^{(0)},\ldots,z^{(n)})[w^{(0)},\ldots,w^{(n)}]$ in terms of $w^{(0)}, \ldots, w^{(n)}$ and $z^{(0)}, \ldots, z^{(n)}$ by substitution into $p=0$. Since $p$ is l.h.o., $w^{(n)}$ appears linearly in $R$. 
From $R$, we deduce the rational expression $r_p\in \KK(x)(z^{(\leq n)},w^{(\leq n-1)})$ for which $w^{(n)}=r_p(z^{(0)},\ldots,z^{(n)},w^{(0)},\ldots,w^{(n-1)})$. We define $u_1=z^{(0)},$ $\ldots,$ $u_m=z^{(m-1)},$ $u_{m+1}=w^{(0)},$ $\ldots,$ $u_{m+n}=w^{(n-1)}$, and deduce $r_q$ from~$q$ as in the proof of \Cref{thm:boundldf}. Furthermore, in case $m \leq n$, we write the $z^{(j)}$'s, $m~\leq~j~\leq~n$, in terms of $z^{(0)},\ldots,z^{(m-1)}$ using $r_q$, so that $r_p$ depends only on $x,u_1,\ldots,u_{n+m}$. We build the following dynamical model over $\KK(x)$ in the indeterminates $u_1,\ldots,u_{m+n},$ and $w$:
\begin{align}\label{eq:dynmodcomp}
\begin{split}
u_1' \,=\, u_2, \ \ldots \ ,\, u_{m-1}'\,=\, u_m, \  u_m' \,=\, r_q(x,u_1,\ldots,u_m),\\
u_{m+1}' \,=\, u_{m+2}, 
\ \ldots , \, u_{m+n-1}' \,=\, u_{m+n}, \
u_{m+n}'\,=\,r_p(x,u_1,\ldots,u_{m+n}),\\ 
w = u_{m+1}.
\end{split}
\end{align}
We obtained a dynamical model of dimension $m+n$.
By \Cref{prop:minp}, $\mathfrak{o}_{\circ}(p,q)\leq m+n$.
\end{proof}

In a similar way as \Cref{cor:boudldf}, we deduce  
\begin{corollary}\label{cor:cboundldf} 
Let $p\in \FF[y^{(\infty)}]$ and $q\in \FF[z^{(\infty)}]$ be of order $n$ and $m$, respectively. Then $\mathfrak{o}_{\circ}(p,q) \, \leq\,  m+n+2 $. \hfill \qed
\end{corollary}

\section{Algorithms and applications}\label{sec:algo}
In this section, we summarize the steps of the algorithms described in the previous sections and apply them to examples from the sciences.

\subsection{Algorithms}\label{sec:algos}
We start with the algorithm that returns a differential polynomial of lowest order w.r.t.\  a dynamical model of  type~\eqref{eq:dynmodel}. After that, we present algorithms for arithmetic binary operations and compositions of D-algebraic functions for Methods I and II. As the description of the unary case is straightforward from these, we omit to present it here. For some explanations, see \Cref{rem:polsols}. 
However, all codes are provided on the Mathrepo.

We implemented the first method in {Macaulay2} using jets for truncating differential ideals. Based on the second method, we initiated the {Maple} package {\tt NLDE}\footnote{{\tt NLDE} is an acronym for ``{\bf N}on{\bf l}inear (algebra) and {\bf d}ifferential {\bf e}quations''.} which, at the time of writing this paper, contains five procedures, namely one for deriving differential polynomials from dynamical models, and four for dealing with compositions, inverse functions, and rational functions of D-algebraic functions.

\Cref{algo:Algo1} is based on \Cref{prop:minp}. We implemented it in Maple as the command {\tt SystoMinDiffPoly}.\footnote{The command is to be used as {\tt SystoMinDiffPoly} or {\tt NLDE:-SystoMinDiffPoly}, depending on the configuration of the user's Maple session with the package {\tt NLDE}. } The main idea behind our second method is to model operations between D-algebraic functions in such a way that \Cref{algo:Algo1} can be applied.

\noindent\makebox[\linewidth]{\rule{\textwidth}{0.4pt}}

\vspace*{-4mm}
\begin{Algorithm}[Algorithm behind \Cref{prop:minp}]\label{algo:Algo1}
    \begin{algorithmic} 
    \\ \Require $\mathbf{A}$ and $B$ from \eqref{eq:dynmodel}. 
    \Ensure A differential polynomial of lowest possible order that vanishes at all $B(\mathbf{f}(x))~\in~\cA$.
    \begin{enumerate}
        \item Let $Q$ be the least common multiple of the denominator of $B$ and the denominators of the $A_i$'s. We define the polynomials $a_i$ and $b$ as in \eqref{eq:dynmodel} such that
        \begin{align*}
        A_i \,=\,\frac{a_i}{Q} \quad \text{and} \quad B \,=\, \frac{b}{Q}\quad\text{for}\quad 1\leq i \leq n \, .
        \end{align*}
        \item Denote by $S$ the set of polynomials 
        \begin{eqnarray*}
             S&\coloneqq &\lbrace Q\,\mathbf{y}'-\mathbf{a}(\mathbf{y}),\, Q \, z - b(\mathbf{y})\rbrace\\
               &=& \lbrace Q\,y_1'-a_1(y_1,\ldots,y_n),\, \ldots,\, Q\,y_n'-a_n(y_1,\ldots,y_n),\, Q\,z - b(y_1,\ldots,y_n)\rbrace 
        \end{eqnarray*}
        in the polynomial ring $\KK[x,y_1,\ldots,y_n,y_1',\ldots,y_n',z]$.
        \item Compute the first $n-1$ derivatives of all polynomials in $S$ and add them to $S$.
        \item Compute the $n$-th derivative of $Q\, z - b(y_1,\ldots,y_n)$ and add it to $S$. This now lives in the ring $\KK[x,y_1^{(\leq n)},\ldots,y_n^{(\leq n)},z^{(\leq n)}]$.
        \item Let  $I\coloneqq\langle S \rangle \subset \KK[x,y_1^{(\leq n)},\ldots,y_n^{(\leq n)},z^{(\leq n)}]$ be the ideal generated by the elements of~$S$.
        \item Saturate $I$ by $Q$, i.e., $I \coloneqq I: Q^{\infty}$.
        \item Compute the elimination ideal $I \cap \KK[x][z^{(\leq n)}]$. From the obtained Gr\"{o}bner basis, choose a non-zero polynomial $r$ of lowest degree among those of the lowest order. 
        \item Return $r$.
    \end{enumerate}	
    \end{algorithmic}
    \vspace*{-4mm}
    
\noindent\makebox[\linewidth]{\rule{\textwidth}{0.4pt}}
\end{Algorithm}

The next algorithm is based on \Cref{prop:directmeth} for arithmetic operations  with D-algebraic functions.

\noindent\makebox[\linewidth]{\rule{\textwidth}{0.4pt}}

\vspace*{-4mm}
\begin{Algorithm}[Method I for  $\alpha \in \{+, -, \times, /\}$  (see \Cref{sec:directmth})]
    \label{algo:Algo2}
    \begin{algorithmic} 
    \\ \Require Two differential polynomials $p\in\KK[x][y^{(\infty)}],q\in \KK[x][z^{(\infty)}]$ of order $n$ and $m$, resp., and $\alpha\in\{+, -, \times, /\}$. 
    \Ensure Differential polynomials that vanish at all $\alpha(f,g)$ s.t. $p(f)=q(g)=0$.
    \begin{enumerate}
        \item Let $R_{\alpha}$ be the numerator of $w-\alpha(y,z)$. 
        \item Let $j=0$.
        \item Let $I_\alpha^{(\leq j)} = \langle p, q, R_\alpha \rangle^{(\leq j)} \subset \KK[x][y^{(\leq j + n)},z^{\leq j + m)}, w^{(\leq j)}]$.
        \item While $I_\alpha^{(\leq j)} \cap \KK[x][w^{(\infty)}] = \langle 0\rangle$ do:   $j\coloneqq j+1$
        \item  Return $I_\alpha^{(\leq j)} \cap \KK[x][w^{(\infty)}]$.
    \end{enumerate}	
    \end{algorithmic}
        \vspace*{-4mm}
    
\noindent\makebox[\linewidth]{\rule{\textwidth}{0.4pt}}
\end{Algorithm}

The next algorithm for the composition of D-algebraic functions is based on Method~I.

\noindent\makebox[\linewidth]{\rule{\textwidth}{0.4pt}}

\vspace*{-4mm}
\begin{Algorithm}[Method I for composition (see \Cref{sec:directmth})]
    \label{algo:Algo3}
    \begin{algorithmic} 
  \\  \Require Two differential polynomials $p\in\KK[x][y^{(\infty)}],q\in \KK[x][z^{(\infty)}]$ of order $n$ and $m$, resp.
    \Ensure A differential polynomial that vanishes at all $\circ(f,g)$ for which $p(f)=q(g)=0$.
    \begin{enumerate}
        \item Let $j = 0$.
        \item Let $I^{(\leq j)} = \langle p, q \rangle^{(\leq j)} \subset \KK[x][y^{(\leq n + j)}, z^{(\leq m + j)}]$, with the substitution in~$p$ as explained right before~\eqref{eq:compositionEqs}.
        \item For $ 0 \leq i \leq j + \min(m,n) $ construct $S_i(y,z)$ as explained in \eqref{eq:compositionEqs} and let $H_j$ be the ideal $I^{(\leq j)} + \langle w^{(i)} - S_i(y,z) \, | \, i \leq j + \min(m,n) \rangle$.
        \item While $H_j \cap \KK[x][w^{(\leq j + \min(m,n))}] = \langle 0\rangle $ do: $j:= j + 1.$
        \item Return a nonzero polynomial in $H_j \cap \KK[x][w^{(\leq j + \min(m,n))})]$.
    \end{enumerate}
    \end{algorithmic}
        \vspace*{-4mm}
    
\noindent\makebox[\linewidth]{\rule{\textwidth}{0.4pt}}
\end{Algorithm}

We implemented \Cref{algo:Algo2} and \Cref{algo:Algo3} in Macaulay2. 

We now give the equivalent of Algorithms \ref{algo:Algo2} and \ref{algo:Algo3} for our second method.

\noindent\makebox[\linewidth]{\rule{\textwidth}{0.4pt}}

\vspace*{-4mm}
\begin{Algorithm}[Method II for $\alpha\in\{+, -, \times, /\}$ (see \Cref{sec:secmeth})]\label{algo:Algo4}
\begin{algorithmic} 
    \\  \Require Two irreducible differential polynomials $p\in\KK[x][y^{(\infty)}],$ $q\in \KK[x][z^{(\infty)}]$ of order $n$ and $m$, resp., and $\alpha\in\{+, -, \times, /\}$.
    \Ensure A differential polynomial of order $N\leq m+n+2$ that vanishes at generic $\alpha(f,g)$ for which $p(f)=q(g)=0$ and $f^{(n+1)}, \, g^{(m+1)}\neq 0.$
    \begin{enumerate}
        \item Let $p_1=p$ if $p$ is l.h.o.\ and $p_1=p'$ otherwise, and define $q_1$ similarly for $q$. The differential polynomials $p_1$ and $q_1$ are l.h.o.
        \item Use $p_1$ and $q_1$ to build the dynamical model \eqref{eq:dynsysalpha}, and denote it by $\mathcal{M}_{\alpha(p_1,q_1)}$.
        \item Return the output of \Cref{algo:Algo1} from the input system $\mathcal{M}_{\alpha(p_1,q_1)}$.
    \end{enumerate}	
    \end{algorithmic}
        \vspace*{-4mm}
    
\noindent\makebox[\linewidth]{\rule{\textwidth}{0.4pt}}
\end{Algorithm}

\noindent\makebox[\linewidth]{\rule{\textwidth}{0.4pt}}

\vspace*{-4mm}
\begin{Algorithm}[Method II for composition (see \Cref{sec:secmeth})]\label{algo:Algo5}
\begin{algorithmic}     
\\ \Require Two irreducible differential polynomials $p\in\KK[x][y^{(\infty)}]$ and $q\in \KK[x][z^{(\infty)}]$ of order $n$ and $m$, respectively.
    \Ensure A differential polynomial of order at most $ m+n+2,$ that vanishes at generic $\circ(f,g)$ for which $p(f)=q(g)=0$ and $f^{(n+1)},\, g^{(m+1)}\neq 0.$
    \begin{enumerate}
        \item Let $p_1=p,$ $ n_1=n$ if $p$ is l.h.o.\ and $p_1=p',$ $ n_1=n+1$ otherwise. In the same way, define $q_1$ and $m_1$ starting from~$q$.
        \item Consider the system of equations $w=y,$ $ w'=z'\,y'$, $w''=z''y'+(z')^2\, y'',$ $\ldots,\, w^{(n)} = S_n(y,z)$, where $S_n\in \KK[x][y^{(\infty)},z^{(\infty)}]$ denotes the differential polynomial defined in~\eqref{eq:compositionEqs}. Via this system, express $y^{(k)}$ in terms of $w,\ldots,w^{(k)},z,\ldots,z^{(k)}$, i.e., $y^{(k)}=a_k$ with $a_k \in \KK(w,\ldots,w^{(k)},z,\ldots,z^{(k)}).$
        \item In $p_1$, substitute $y^{(k)}$ by the expression encoded by $a_k$ and substitute $x$ by $z$. This results in a rational function $R\in \KK(w,\ldots,w^{(k)},z,\ldots,z^{(k)})$.
        \item In case $m_1 \leq n_1$, substitute all occurrences of $z^{(j)}$, $m_1\leq j \leq n_1$ in $R$ by their expressions in terms of $z,\ldots, z^{(m_1-1)}$ using the algebraic relation encoded by~$q$.
        \item Use $q_1$ and $R$ to build the dynamical model \eqref{eq:dynmodcomp} and denote it by $\mathcal{M}_{\circ(p,q)}$.
        \item Return the output of \Cref{algo:Algo1} from the input system $\mathcal{M}_{\circ(p,q)}$.
    \end{enumerate}	
    \end{algorithmic}
        \vspace*{-4mm}
    
\noindent\makebox[\linewidth]{\rule{\textwidth}{0.4pt}}
\end{Algorithm}

The correctness of Algorithms \ref{algo:Algo1}, \ref{algo:Algo2}, \ref{algo:Algo3}, \ref{algo:Algo4}, and \ref{algo:Algo5} follows from combining \Cref{prop:minp} with Sections~\ref{sec:directmth} and~\ref{sec:secmeth}.

\medskip

Our Maple package {\tt NLDE} contains the procedures {\tt arithmeticDalg} and {\tt composeDalg} for rational expressions and compositions of D-algebraic functions, respectively. The setting is for differential equations. They take three arguments. The {\tt arithmeticDalg} command takes as first argument a list of differential equations, say ${\tt DE1}(y_1(x)),\ldots,{\tt DEN}(y_N(x))$, the second is the list of their dependent variables $y_1(x),\ldots,y_N(x)$ in the same order as in the first input, and the third argument is an equation of the form $z=r(x,y_1,\ldots,y_N)$, where $r$ is a rational function in $x$ and the $y_i$'s. The output is a differential polynomial in $z$. 
Thus, addition, multiplication, and division are particular cases in {\tt arithmeticDalg}. The syntax for {\tt composeDalg} follows the same pattern but with only two elements in each list, and the third argument is the dependent variable for the output. The package also contains the command {\tt unaryDalg} for rational, unary operations with D-algebraic functions.
\begin{example}
Let $p=(y')^3+y+1$ and $q=(y')^2-y-1$.
Using our second method for $\alpha=+$, we deduce the ODE
\begin{align}\label{eq:eqh}
24\, \left( y''(x)\right)^3 - 36\, \left( y''(x)\right)^2 + 18\, y''(x) - 8\,y^{(3)}(x) - 3 \,=\,0 
\end{align}
of order~$3$ from running
\begin{lstlisting}
> ADE1 := diff(y(x),x)^3+y(x)+1=0:
> ADE2 := diff(z(x),x)^2-z(x)-1=0:
> NLDE:-arithmeticDalg([ADE1,ADE2],[y(x),z(x)],w=y+z)
\end{lstlisting}
in Maple.
\end{example}

\begin{remark}[Comparison of Methods I and II]\label{re:compareM}
We here give some remarks to compare Methods I and II. In some cases, there are subtle differences between the differential ideals considered in Methods I and II. The ideal $I_{\alpha}^{(\leq j)}$ from Method $I$ contains the ideal $I_{\mathcal{M}_{\alpha},j}$ from \Cref{eq:idealynsys} without the saturation carried out. Moreover, if both $p$ and $q$ are l.h.o., $I_{\alpha}^{(\leq j)}$ is contained in the saturated ideal~$I_{\mathcal{M}_{\alpha},j}$. Indeed, the two methods perform differently in some cases. Consider, for instance, the two differential polynomials
\begin{align}\label{Intro_DE}
p \,=\, y''y-\left(y'\right)^2 \quad \text{and} \quad q\,=\, \left(y'\right)^2 + y^2 + 1 \, .
\end{align}
Our second method finds the differential polynomial
\begin{align}\label{eq:solintro}
-y\, {y''} - y\, {y^{(4)}} + \left(y'\right)^2 + 2\, {y'}\, {y^{(3)}} - \left(y''\right)^2 - {y''}\, {y^{(4)}} + \big(y^{(3)}\big)^2
\end{align}
for $\alpha=+$ in less than a second, whereas the computations based on the first method in {Macaulay2} did not terminate even after a whole day. A Maple implementation of the first method enabled us to find a third-order differential polynomial of degree $18$ in about half a minute. In our experiments, Method I in general returned ADEs of lower order and higher degree than Method II in the non-l.h.o.\ case, and both methods seem to perform similarly in the l.h.o.\ case. \Cref{table:comp} collects the outputs of Method~I and II for the operations $\alpha \in \{+,\times,/,\circ\}$ for the following eight differential polynomials:  
\vspace*{-8mm}
\begin{multicols}{2}
\begin{align}
   y'\,- x\,y^2 , \label{cADE11}\\
    x\,y'-x^2+y-1\, ,\label{cADE12}\\
    y'\,y + y''\, ,\label{cADE13}\\
    y^3-y^{(3)} ,\label{cADE14}
    \end{align}
    
    \columnbreak
    
    \begin{align}
    -\left(z'\right)^2+z+x+1 \, ,\label{cADE21}\\
    z\,z'+3\,z'+2\,x^2+2\label{cADE22}\, ,\\
    z' + x\,z''\, ,\label{cADE23}\\
    z'-z^2\, .\label{cADE24}
\end{align} 
\end{multicols}

\begin{table}[ht]
\begin{tabular}{|l|c|c|c|c|}
\hline
\begin{tabular}{c}
\end{tabular}
& \eqref{cADE11} $\alpha$ \eqref{cADE21}
& \eqref{cADE12} $\alpha$ \eqref{cADE22} 
& \eqref{cADE13} $\alpha$ \eqref{cADE23}
& \eqref{cADE14} $\alpha$ \eqref{cADE24} 
\\\hline
$\alpha =+ $   & \begin{tabular}[c]{@{}c@{}}(I, 364.185, 2, 8) \\ (II, 2{,}152.547, 3, 13)\end{tabular} 
                & \begin{tabular}[c]{@{}c@{}}(I, 3{,}000+, -, -)\\ (II, 0.218, 2, 4)\end{tabular}                                             
                & \begin{tabular}[c]{@{}c@{}}(I, 3{,}000+, -, -)\\ (II, 5.688, 4, 6)\end{tabular}                                             
                & \begin{tabular}[c]{@{}c@{}}(I, 3{,}000+, -, -) \\ (II, 30.281, 4, 15)\end{tabular}                                             
                \\\hline
$\alpha =\times $ & \begin{tabular}[c]{@{}c@{}}(I, 3{,}000+, -, -)\\ (II, 345.453, 3, 8)\end{tabular}                                              
                & \begin{tabular}[c]{@{}c@{}}(I, 3{,}000+, -, -) \\ (II, 0.641, 2, 5)\end{tabular}                                       
                & \begin{tabular}[c]{@{}c@{}}(I, 3{,}000+, -, -) \\ (II, 51.343, 4, 7)\end{tabular}                                        
                & \begin{tabular}[c]{@{}c@{}}(I, 3{,}000+, -, -)\\  (II, 1.547, 4, 10)\end{tabular}                                    
                \\\hline
$\alpha = / $ & \begin{tabular}[c]{@{}c@{}}(I, 3{,}000+, -, -)\\ (II, 3{,}000+, -, -)\end{tabular}                                            
                & \begin{tabular}[c]{@{}c@{}}(I, 3{,}000+, -, -)\\ (II, 44.953, 2, 5)\end{tabular}                                             
                & \begin{tabular}[c]{@{}c@{}}(I, 3{,}000+, -, -)\\ (II, 3{,}000+, -, -)\end{tabular}                                             
                & \begin{tabular}[c]{@{}c@{}}(I, 3{,}000+, -, -)\\ (II, 4.453, 4, 11)\end{tabular}                                    
                \\\hline
$\alpha =\circ $ & \begin{tabular}[c]{@{}c@{}}(I, 0.127, 2, 11)\\ (II, 41.985, 3, 18)\end{tabular}                                              
                & \begin{tabular}[c]{@{}c@{}}(I, 0.314, 2, 5)\\ (II, 37.984, 2, 5)\end{tabular}                                             
                & \begin{tabular}[c]{@{}c@{}}(I, 0.117, 3, 3)\\ (II, 1.797, 3, 3)\end{tabular}                                           
                & \begin{tabular}[c]{@{}c@{}}(I, 3{,}000+, -, -)\\ (II, 51.531, 4, 16)\end{tabular}                                      
                \\\hline
\end{tabular}
\caption{Outputs (I/II, $t, \, r,\, d\,$) of Methods I and II for the differential polynomials in \eqref{cADE11}--\eqref{cADE24}. We here display the CPU time~$t$ in seconds as well as the order~$r$ and the degree~$d$ of the computed differential polynomial. In case we did not obtain a result after $3{,}000$ seconds, we stopped the computation. This is indicated as ``$3{,}000+$'' in the table.}  
\label{table:comp}
\end{table}

An advantage of Method I is that it does not require any genericity assumption on the considered solutions. Yet, since this method builds on Gr\"{o}bner bases, the runtime is expected to scale badly when the degree and number of variables increase. For more details about the complexity analysis, we refer to~\cite{joswig2018degree} and the references~therein.
\end{remark}
\begin{remark}[Polynomial solutions]\label{rem:polsols}
We now present how one can deal with polynomial solutions of low degree. Let $p=y''y-(y')^2$ and $q=(y')^2+y^2+1$, as in \Cref{re:compareM}. The only problematic polynomial zero of $p$ is the zero polynomial, and for $q$ the constant functions~$\pm i$. 
Then $f\pm i$, where $f$ is a zero of~$p$, is not a solution of~\eqref{eq:solintro}. This can be seen for $f=\exp$, for instance. We find differential polynomials for these two cases through the unary operations $y\pm i$ over $\QQ(i)$, and we get 
$\pm i\, {y''} + {y}\, {y''} - \left(y'\right)^2 \, .$
Alternatively, one can use any of our two methods with the algebraic equation of $\pm i$ as an ADE of order $0$ as second input.
\end{remark}

\subsection{Applications}\label{sec:mDalgApp}
We here showcase computations for some D-algebraic functions from practical applications, among others from high energy physics and biology, and also prove some functional identities.
\begin{example}[Elliptic functions]\label{ex:Feynman}
Again, let $\wp$ denote the Weierstrass elliptic function. It fulfills 
\begin{align}\label{eq:ADEwp1}
\left(\wp'(x)\right)^2 \,=\, 4 \left(\wp(x)\right)^3 - g_2 \wp(x) -g_3 \, ,
\end{align}
where $g_2$ and $g_3$ are constants depending on the periods of~$\wp$. Now consider the following rational function of the Weierstrass elliptic function:
\begin{align}\label{eq:kappa1}
\kappa (x) \,=\, \frac{-3a_1a_{13}a_{24}\wp(x)+a_1^2s_1(a_2,a_3,a_4)-2a_1s_2(a_2,a_3,a_4)+3s_3(a_2,a_3,a_4)}{-3a_{13}a_{24}\wp(x)+3a_1^2-2a_1s_1(a_2,a_3,a_4)+s_2(a_2,a_3,a_4)} \,.
\end{align}
It arises in study of Feynman integrals on elliptic curves in~\cite[Section~7.1]{ellpolylog}. 
In~\eqref{eq:kappa1}, $s_n$ denotes the elementary symmetric polynomial of degree $n$ in three variables, and $a_{ij}=a_i-a_j$. The following equation was given in \cite[Equation (7.12)]{ellpolylog} as an ADE satisfied by~$\kappa$:
\begin{align}\label{eq:ADEkappa1}
    \left(\frac{\sqrt{\left(a_1-a_3\right)\,\left(a_2-a_4\right)}}{2}\, \kappa'(x)\right)^2 \,=\, \left(\kappa(x)-a_1\right)\,\left(\kappa(x)-a_2\right)\,\left(\kappa(x)-a_3\right)\,\left(\kappa(x)-a_4\right).
\end{align}
To computationally verify this equation, we aim to express the invariants $g_2$ and $g_3$ in terms of $a_1,a_2,a_3,$ and $a_4$. We use our implementation of Method~II to compute an ADE of the same expanded form of \eqref{eq:ADEkappa1} and write it in the following way:
\begin{align}\label{eq:ADEkappa2}
\left(\frac{\sqrt{\left(a_1-a_3\right)\,\left(a_2-a_4\right)}}{2}\, \kappa'(x)\right)^2 \,=\  C_4\, {\kappa(x)}^4 + C_3\, {\kappa(x)}^3 + C_2\, {\kappa(x)}^2 + C_1\, {\kappa(x)} + C_0 \, ,
\end{align}
where $C_0,\ldots,C_4$ are rational expressions in $a_1,a_2,a_3,a_4,g_2,g_3$. By comparison of the coefficients of powers of $\kappa(x)$ in~\eqref{eq:ADEkappa1} and~\eqref{eq:ADEkappa2}, we construct the system of equations
\begin{align}\label{eq:systemg2g3}
    \begin{cases}
    C_0\left(a_1,a_2,a_3,a_4,g_2,g_3\right) \,=\, a_1\,a_2\,a_3\,a_4 \, ,\\
    C_1\left(a_1,a_2,a_3,a_4,g_2,g_3\right) \,=\, - \left( a_1\,a_2\,a_3 + a_1\,a_2\,a_4 + a_1\,a_3\,a_4 + a_2\,a_3\,a_4\right)  ,\\
    C_2\left(a_1,a_2,a_3,a_4,g_2,g_3\right) \,=\, a_1\,a_2 + a_1\,a_3 + a_1\,a_4 + a_2\,a_3 + a_2\,a_4 + a_3\,a_4 \, ,\\
    C_3\left(a_1,a_2,a_3,a_4,g_2,g_3\right) \,=\, -\left(a_1+a_2+a_3+a_4\right)\, ,\\
    C_4\left(a_1,a_2,a_3,a_4,g_2,g_3\right) \,=\, 1 \, .
\end{cases}
\end{align}
Solving the system for $g_2$ and $g_3$ yields
\begin{small}
\begin{align}\begin{split}
 g_2 \,=\, \frac{4}{{3 \left(a_2-a_4\right)^{2} \left(a_1-a_3\right)^{2}}}\Big(a_1^{2} a_2^{2}-a_1^{2} a_2 a_3-a_1^{2} a_2 a_4+a_1^{2} a_3^{2}-a_1^{2} a_3 a_4+a_1^{2} a_4^{2}-a_1 \,a_2^{2} a_3\hspace*{3.3mm}\\
    \phantom{abcdefghijklmn}-a_1 \,a_2^{2} a_4-a_1 a_2 \,a_3^{2}+6 a_1 a_2 a_3 a_4-a_1 a_2 \,a_4^{2}-a_1 \,a_3^{2} a_4-a_1 a_3 \,a_4^{2}\hspace*{3.3mm}\\
    +a_2^{2} a_3^{2}-a_2^{2} a_3 a_4+a_2^{2} a_4^{2}-a_2 \,a_3^{2} a_4-a_2 a_3 \,a_4^{2}+a_3^{2} a_4^{2} \Big),
\end{split}\end{align}
\end{small}
and
\begin{small}
\begin{align}\begin{split}
 g_3 \,=\, -\frac{4}{27 \left(a_2-a_4\right)^{3} \left(a_1-a_3\right)^{3}}\Big( \left(a_1 a_2+a_1 a_3-2 a_1 a_4-2 a_2 a_3+a_2 a_4+a_3 a_4\right)\hspace*{4.3mm}\\
    \phantom{abcdefghijklmn}\left(2 a_1 a_2-a_1 a_3-a_1 a_4-a_2 a_3-a_2 a_4+2 a_3 a_4\right)\hspace*{4.3mm}\\
    \left(a_1 a_2-2 a_1 a_3+a_1 a_4+a_2 a_3-2 a_2 a_4+a_3 a_4\right)\Big).
\end{split}\end{align}
\end{small}

\noindent In summary, starting from the ADE~\eqref{eq:ADEkappa1} for~$\kappa$, we reconstructed the invariants $g_1$ and $g_2$ of the Weierstrass elliptic function, of which~$\kappa$ is a rational function of.
\end{example}

\begin{example}[Epidemiology]\label{ex:SIR} We here demonstrate how to compute an ADE in practice, starting from a dynamical system.
The epidemic stage of the SIR (Susceptible-Infected-Removed) model is given by the system of differential equations:
\begin{align}
\begin{split}
S' &\,\,=\,\, -\beta\, S\, I - \delta\, S + \mu \, ,\\
I' &\,\,=\,\, \beta\, S\, I - \gamma\, I + \nu \, ,\\
R' &\,\,=\,\, \delta\, S + \gamma\, I \,.
\end{split}
\end{align}
This model describes how a disease can spread within a population.
We refer our readers to~\cite{scharnhorst2012models} for details about the parameters $\delta,\gamma,\mu,\nu$. The derivation used is $\sfrac{\partial}{\partial t}$, where $t$ represents the time. Suppose we want to compute the minimal differential equation for $R$. Then {\tt SystoMinDiffPoly} will take the triple
\begin{align}\label{eq:dervarSIR}
\left[ -\beta\, S\, T - \delta\, S + \mu,\ \beta\, S\, T - \gamma\, T + \nu, \ \delta  \,  S + \gamma\, T \right]
\end{align}
as first argument, where $I$ is replaced by $T$ because the variable $I$ is protected for the imaginary number; as second argument, the function depending on $S, T,$ and $R$ for which we seek a differential equation; the third argument is $[S,T,R]$ which represents the main functions of the system whose derivatives are given in the same order in \eqref{eq:dervarSIR}; and finally, the dependent variable for the sought differential equation, we choose $f(t)$ since $R$ is already used among the variables of the system. So the syntax together with the output~is:
\begin{lstlisting}
> timing,p := CPUTime(NLDE:-SystoMinDiffPoly([-beta*S*T-delta*S + mu, beta*S*T - gamma*T + nu, delta*S+gamma*T],R,[S,T,R],f(t))):timing
\end{lstlisting}
\vspace*{-3mm}
\begin{align*}
0.485
\end{align*}
\vspace*{-10mm}
\begin{lstlisting}
> PDEtools:-difforder(p,t)
\end{lstlisting}
\vspace*{-4mm}
\begin{align*}
3
\end{align*}
The output differential equation of our Maple implementation, here {\tt p}, is of order $3$ and degree~$4$. We here do not display~$p$, as it requires about ten lines. We made them available as supplementary files at \href{https://mathrepo.mis.mpg.de/DAlgebraicFunctions/DAlgebraicFunc-Examples-Maple.html}{\tt https://mathrepo.mis.mpg.de/DAlgebraicFunctions}.
Note, however, that the code above computes the differential polynomial in about half a second. Such equations are important to study certain states of the input system.
\end{example}

\begin{example}[Painlev\'{e} transcendent I]
With our second method, we will compute third-order differential equations for the exponential and the square root of the Painlev\'{e} transcendent of type~I. These equations are relevant for understanding important questions related to Painlev\'{e} transcendents, cf.~\cite[Sections 1 \& 2]{clarkson2019open}. Their solutions are particularly interesting special functions, since they cannot be expressed in terms of elementary functions or well-known special functions. We will use $y'-y$ and $2xy'-y$ as input for $\exp$ and $\sqrt{\cdot}$, respectively.
For the exponential, we~run 
\begin{lstlisting}
> ADE1 := diff(y(x), x) - y(x) = 0:
> ADE2 := diff(z(x),x,x)=6*z(x)^2+x: #the transcendent
> NLDE:-composeDalg([ADE1,ADE2],[y(x),z(x)],w(x))
\end{lstlisting}
to obtain the ADE
{\footnotesize 
\begin{align*}
\hspace*{-11mm} {24 x \left(w'(x)\right)^{2} \left(w \! \left(x \right)\right)^4+\left(w \! \left(x \right)\right)^{6}-2 \left(w \! \left(x \right)\right)^{5} w^{(3)}(x)+6 w''(x)  w'(x)  \left(w \! \left(x \right)\right)^4}
+\big(w^{(3)}(x)\big)^{2} \left(w \! \left(x \right)\right)^4 \\ 
\hspace*{-11mm} -4 \left(w'(x)\right)^3-24 w''(x) \left(w'(x)\right)^{2} \left(w \! \left(x \right)\right)^3
-6 w^{(3)}(x) w''(x)  w'(x)  \left(w \! \left(x \right)\right)^3+24 \left(w'(x)\right)^4 w \! \left(x \right)^{2}\\
\hspace*{-11mm} {+4 w^{(3)}(x) \left(w'(x)\right)^3 \left(w \! \left(x \right)\right)^{2}+9 \left(w''(x)\right)^{2} \left(w'(x)\right)^{2} \left(w \! \left(x \right)\right)^{2}}
-12 w''(x) \left(w'(x)\right)^4 w \! \left(x \right)+4 \left(w'(x)\right)^{6}\,=\, 0 \, .
\hspace*{-11mm}
\end{align*}}

\noindent For the square root, we run
\begin{lstlisting}
> ADE3 := 2*x*diff(y(x), x) - y(x) = 0:
> NLDE:-composeDalg([ADE3,ADE2],[y(x),z(x)],w(x))
\end{lstlisting}
to obtain the ADE
{\footnotesize
\begin{align*}
\hspace*{-9.5mm} -48 x^{2} \left(w'(x)\right)^2 \left(w \! \left(x \right)\right)^3+24 x \left(w \! \left(x \right)\right)^4 w'(x) -2 x \big(w^{(3)}(x)\big)^2 \left(w \! \left(x \right)\right)^3
-4 x w''(x)  w^{(3)}(x)  w'(x)  \left(w \! \left(x \right)\right)^2 \\
\hspace*{-9.5mm} + 8 x \left(w'(x)\right)^3 w^{(3)}(x)  w \! \left(x \right) 
+6 x \left(w'(x)\right)^2 \left(w''(x)\right)^2 w \! \left(x \right)+24 x \left(w'(x)\right)^4 w''(x)
+2 w''(x)  w^{(3)}(x)  \left(w \! \left(x \right)\right)^3 \\
\hspace*{-9.5mm}  -3 \left(w \! \left(x \right)\right)^{5} + 2 \left(w'(x)\right)^2 w^{(3)}(x) \left( w \! \left(x \right)\right)^2 
-2 \left(w''(x)\right)^2 w'(x)  \left(w \! \left(x \right)\right)^2-10 \left(w'(x)\right)^3 w''(x)  w \! \left(x \right)
{-8 \left(w'(x)\right)^{5} = 0 } \hspace*{-9mm}
\end{align*}}

\noindent of order $3$ and degree $5$.
\end{example}

As a further application of our methods, we are now going to prove functional identities.

\begin{example}
\label{ex:duplication}
Consider the two trigonometric identities
\begin{align}
		   \tan(3x) &\,=\, \frac{3\,\tan(x)-\tan^3(x)}{1-3\,\tan^2(x)}\, ,\label{eq:id1}\\
		   \sec(3x) &\,=\, \frac{\sec^3(x)}{4-3\,\sec^2(x)}\label{eq:id2}\, ,
\end{align}
with $x$ taken from appropriate domains. For each of the identities of the form $f=g$ in~\eqref{eq:id1} and~\eqref{eq:id2}, we will show that $f$ and $g$ satisfy the same differential equation. Since all the functions are analytic, proving the functional identities reduces to comparing sufficiently many of their series coefficients; the latter are called ``initial values'' in this setup (see \cite{TeguiaDelta2,vdH19}). 
We start from the following ADE fulfilled by $\tan(x)$:
\begin{equation}\label{eq:tan}
		     t'(x) \,=\, \left(t(x)\right)^2+1 \, .
\end{equation}
For the left-hand side of \eqref{eq:id1}, we compute an ADE for the composition $\tan \, \circ \, \left(x\mapsto 3 x\right)$. An ODE satisfied by $x\mapsto 3 x$ is $y'(x)=3$. We use \texttt{NLDE:-composeDalg}: 
\begin{lstlisting}
> ADE1 := diff(t(x),x)=t(x)^2+1: #ADE fulfilled by tan(x)
> NLDE:-composeDalg([ADE1,diff(y(x),x)=3],[t(x),y(x)],z(x))
\end{lstlisting}
and find that
\begin{equation}\label{eq:id1lhs}
		 -3 \left(z \! \left(x \right)\right)^{2}+\frac{d}{d x}z \! \left(x \right)-3 \,=\, 0\, .
		 \end{equation}
For the right-hand side of \eqref{eq:id1}, we use \texttt{NLDE:-unaryDalg}. 
\begin{lstlisting}
> NLDE:-unaryDalg(ADE1,t(x),z=(3*t-t^3)/(1-3*t^2))
\end{lstlisting}
This returns the same ADE as in \eqref{eq:id1lhs}.
Hence, one concludes that the identity \eqref{eq:id1} holds after checking that the $3$rd-order truncations of the Taylor series of $\tan(3\,x)$ and $(3\,\tan(x)-\tan^3(x))/(1-3\,\tan^2(x))$ are identical. By plugging in an arbitrary formal power series into~\eqref{eq:id1lhs} 
and equating the coefficients, one symbolically verifies that the entire sequence of coefficients of any power series solution to~\eqref{eq:id1lhs} is uniquely determined by $3$ initial values and a recursive formula, cf.~\cite{TeguiaDelta2}. 
	
To prove identity~\eqref{eq:id2}, we exploit \eqref{eq:tan} and the following ADE fulfilled by $\cos(x)$:
\begin{equation}
\left(c'(x)\right)^2+\left(c(x)\right)^2 \,=\, 1 \, .
\end{equation}
Using \texttt{NLDE:-unaryDalg} as
\begin{lstlisting}
> ADE2 := diff(c(x), x)^2 + c(x)^2 = 1:
> ADE3 := NLDE:-unaryDalg(ADE2,c(x),s=1/c)
\end{lstlisting}
we find the following ADE fulfilled by $\sec(x)=1/\cos(x)$:
	 \begin{equation}
		 \operatorname{ADE3} \,\coloneqq \, \left(s \! \left(x \right)\right)^{4}-\left(s \! \left(x \right)\right)^{2}-\left(\frac{d}{d x}s \! \left(x \right)\right)^{2} \,=\, 0\, .
		 \end{equation}
We compute an ADE for the left-hand side of \eqref{eq:id2} with \texttt{NLDE:-composeDalg} via
\begin{lstlisting}
> NLDE:-composeDalg([ADE3,diff(y(x),x)=3],[s(x),y(x)],z(x))
\end{lstlisting}
and get
\begin{equation}\label{eq:lhsid2}
		 -18 \left(z \! \left(x \right)\right)^{3}+9 z \! \left(x \right)+\frac{d^{2}}{d x^{2}}z \! \left(x \right)\,=\, 0\, .
\end{equation}
For the right-hand side  of \eqref{eq:id2}, we use \texttt{NLDE:-unaryDalg} with~\eqref{eq:tan}
\begin{lstlisting}
> ADE4 := NLDE:-unaryDalg(ADE3,s(x),z=s^3/(4-3*s^2))
\end{lstlisting}
to obtain
\begin{equation}
		 \operatorname{ADE4} \,\coloneqq \, 9 \left(z \! \left(x \right)\right)^{4}-9 \left(z \! \left(x \right)\right)^{2}-\left(\frac{d}{d x}z \! \left(x \right)\right)^{2} \,=\, 0 \, .
\end{equation}
We obtain a first-order ADE. We take a derivative of it and factor the result. Running
\begin{lstlisting}
> factor(diff(ADE5,x))
\end{lstlisting}
yields
\begin{equation}
		 -2 \left(\frac{d}{d x}z \! \left(x \right)\right) \left(-18 \left(z \! \left(x \right)\right)^{3}+9 z \! \left(x \right)+\frac{d^{2}}{d x^{2}}z \! \left(x \right)\right) \,=\, 0 \, .
\end{equation}
The left-hand side of \eqref{eq:lhsid2} is one of the factors. As before, the equality of the first three Taylor coefficients of both functions in \eqref{eq:id2} proves that the claimed identity \eqref{eq:id2}~holds.
\end{example}

\begin{example}[Proving the duplication formula of the Weierstrass $\wp$ function] As already seen, Weierstrass' elliptic function is a zero of the differential polynomial
\[p \,\coloneqq\, \left( y' \right)^2-4\,y^3+g_2\,y+g_3 \,.\]
We want to prove that {\em any} non-polynomial zero $\wp(x)$  of $p$ satisfies the following identity:
\begin{equation}\label{eq:dupformwp}
    \wp(2x) \,=\, \frac{1}{4}\left(\frac{\wp''(x)}{\wp'(x)}\right)^2 - 2\,\wp(x) \, .
\end{equation}
For the left-hand side, we use  \Cref{algo:Algo5} for the composition $\wp \circ g$, where $g$ is a zero of the differential polynomial $y'-2$. We obtain the differential polynomial
\[q \,\coloneqq\, y'' - 24\,y^2+2\,g_2 \]
by running the following Maple code:
\begin{lstlisting}
> ADEwp := diff(p(x),x)^2=p(x)^3-g2*p(x)-g3: #ADE for Weierstrass
> NLDE:-composeDalg([ADEwp,diff(y(x),x)=2],[p(x),y(x)],r(x))
\end{lstlisting}
For the right-hand side, let  now $\wp(x)$ be a non-polynomial
zero of~$p$
and use $p$ to express the occurring derivatives of $\wp$ in terms of $\wp$ by observing that
\begin{align*}
    p\left(\wp(x)\right)=0 &\iff \left(\wp'(x)\right)^2 = 4\,\left(\wp(z)\right)^3 - g_2\,\wp(z) - g_3 \, ,\\
    p'\left(\wp(x)\right)=0 &\iff \wp''(x) = 6\,\wp(x) - \frac{g_2}{2} \, .
\end{align*}
Hence, the left-hand side in~\eqref{eq:dupformwp} can be rewritten as 
\begin{equation}\label{eq:rhsdupf}
     \frac{\left(6\,\wp(x) - \frac{g_2}{2}\right)^2}{4\left( 4\,\left(\wp(x)\right)^3 - g_2\,\wp(x) - g_3\right)} - 2\,\wp(x).
\end{equation}
Using our algorithm for unary arithmetic, we find a second-order differential polynomial of degree~$20$, which we denote by~$r$. 
It can be obtained by running the following Maple code:
\begin{lstlisting}
> ADErhs := NLDE:-unaryDalg(diff(ADEwp, x), p(x), 
r = 1/4*(6*p^2 - g2/2)^2/(4*p^3 - g2*p - g3) - 2*p, ordering = lexdeg)
\end{lstlisting}
Finally, we prove the identity in \eqref{eq:dupformwp} by showing that  
\[r \in \langle q\rangle^{(\infty)} . \]
To do so, we show that the differential reduction of $r$ with respect to $q$ is zero. In Maple, this can be verified by simple substitution and differentiation or with one of the packages \texttt{DifferentialAlgebra} and \texttt{DifferentialThomas}. For convenience, we use substitution and differentiation. The first step is to eliminate the second-order term from $r$ by replacing all appearances of the second derivative in $r$ with its expression in terms of $y$ from $q$. With the above Maple codes, this corresponds to the following:
\begin{lstlisting}
> simplify(subs(diff(r(x), x, x) = 
solve(ADEcomp, diff(r(x), x, x)), ADErhs))  
\end{lstlisting}
We obtain the differential polynomial
\begin{align}\begin{split}\label{eq:Tfactor}
-256 \Big(-4 y^{3}+g_2 y+\frac{\left(y'\right)^{2}}{4}+g_3\Big)^{4} \Big(3072 g_2 y^{6}+6912 g_3 y^{5}-624 g_2^{2} y^{4}\quad\\ -1824 g_2 g_3 y^{3}+\left(24 g_2^{3}-432 g_3^{2}\right) y^{2}+120 g_2^{2} g_3 y+g_2^{4}+48 g_2 \,g_3^{2}\Big) \,.
\end{split}\end{align}
Let $t\coloneqq -4 y^{3}+g_2 y+\frac{\left(y'\right)^{2}}{4}+g_3$. Observe that $t^4$ is a factor in \eqref{eq:Tfactor}. To conclude, one verifies that 
$t' = y'/2 \cdot q.$
We mention that the proof can be shortened with the improvement of Method~II from \cite[Section~2.2]{teguia2023arithmetic}, which directly finds $4t$ as the differential polynomial for the right-hand side in \eqref{eq:dupformwp}.
\end{example}

\subsection*{Acknowledgments}
We thank Manuel Kauers, Gleb Pogudin, and Bernd Sturmfels for insightful discussions, and Marc H\"{a}rk\"{o}nen for help with our implementation in Macaulay2. We thank Martijn Hidding for pointing out \Cref{ex:Feynman} to us, and the referees for their helpful comments.
ALS~was partially supported by the Wallenberg~AI, Autonomous Systems and Software Program~(WASP) funded by the Knut and Alice Wallenberg Foundation.

{ 

}

%\newpage
\appendix
\section{Algebraic varieties and jet schemes}\label{app:algvar}
In this appendix, we give a very first glimpse into algebraic varieties and jet schemes. 

\subsection{Algebraic varieties}
Our primary reference for this first part is~\cite{CompAlgGeom}, which puts a focus on computational aspects of algebraic geometry. We refer to 
textbooks such as \cite{FultonCurves,Hartshorne} for diving deeper into the topic. 

\begin{figure}[h]
  \includegraphics[width=4.5cm]{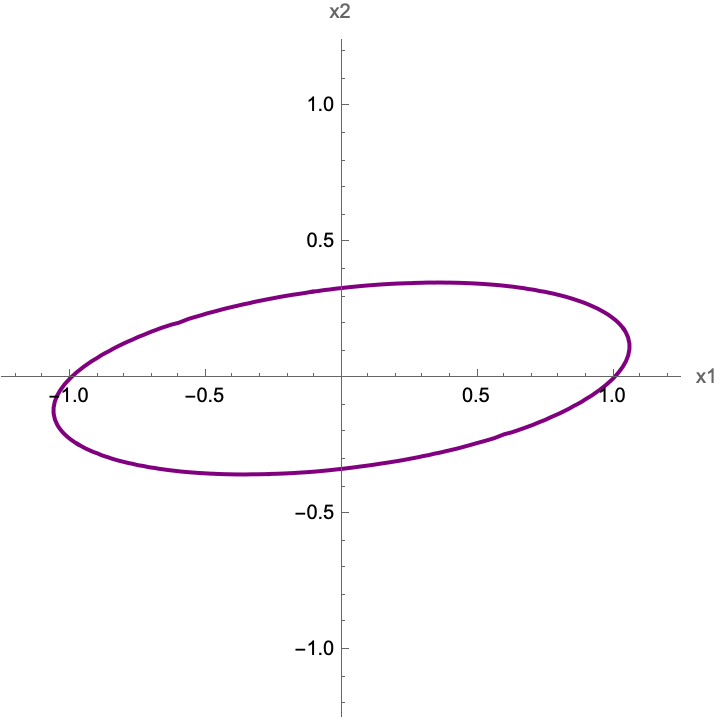} \  \includegraphics[width=4.8cm]{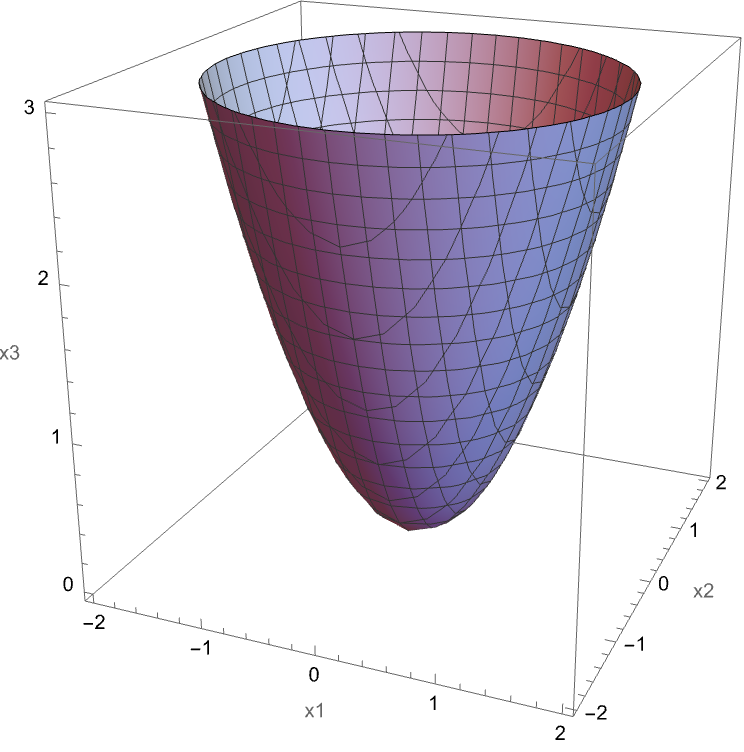}\ \includegraphics[width=5.7cm]{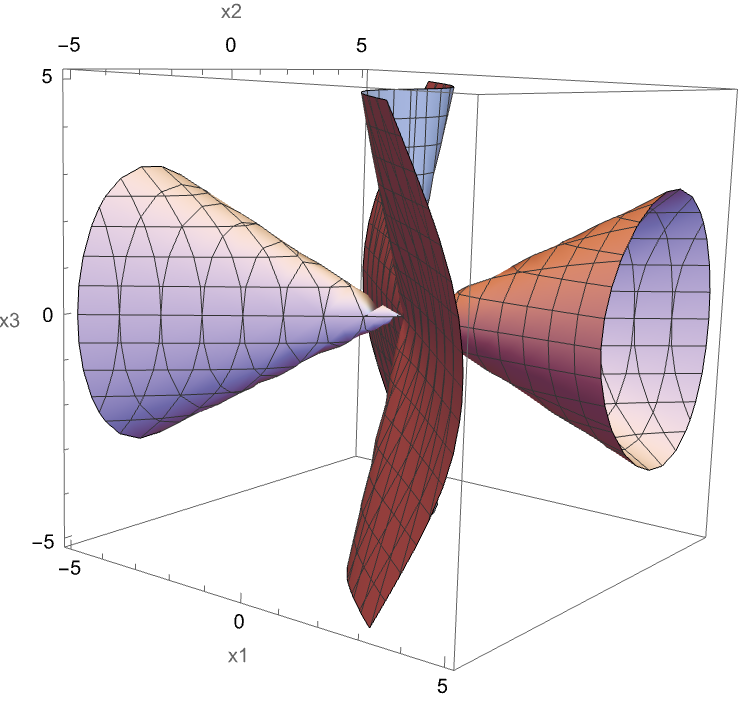}
  \caption{The varieties $V(x_1^2 - 2x_1x_2 + 9x_2^2 - 1)$, $V(x_3-x_1^2-x_2^2)$, and $V(x_1^2x_2 - x_2^3 - x_1x_3^2)$, plotted with Mathematica~\cite{Mathematica}.}
\label{fig:affinevarieties}
\end{figure}

Let $k$ be an arbitrary field, and $f_1,\ldots,f_{\ell}\in k[x_1,\ldots,x_n]$ polynomials in $n$ variables. Their {\em affine variety} $V(f_1,\ldots,f_\ell)\subset k^n$ is the common zero set of $f_1,\ldots,f_\ell$, i.e.,
\begin{align}\label{eq:defvar}
    V\left(f_1,\ldots,f_\ell\right) \, =\,  \left\{ a =(a_1,\ldots,a_n)\in k^n \,|\, f_i(a)=0 \ \text{ for all } \,  i=1,\ldots,\ell \right\} .
\end{align}
Similarly, for $I\subset k[x_1,\ldots,x_n]$ an ideal, its variety is $V(I)=\{ a\in k^n \,| \, f(a)=0 \ \forall f\in I\}$. The solutions of the resulting system of polynomial equations over a $k$-algebra~$K$ constitute the {\em $K$-rational points} of the variety. 
Some examples of affine varieties for the field $k=\RR$ and $\ell=1$ are shown in \Cref{fig:affinevarieties}, namely an ellipse, a paraboloid of revolution, and a cubic surface. If $V=V(f_1,\ldots,f_{\ell})$ and $W=V(g_1,\ldots,g_m)$ are affine varieties, then so are $V\cap W=V(f_1,\ldots,f_{\ell},g_1,\ldots,g_m)$ and $V\cup W=V(f_ig_j \,|\, i=1,\ldots,\ell,\, j=1,\ldots,m)$.

The other way round, given an affine variety $V\subset k^n$, one should ask for the set of {\em all} polynomials that vanish on the variety. They form an ideal $I(V)$ in $k[x_1,\ldots,x_n]$,
\begin{align}
I(V) \, = \, \left\{ f \in k[x_1,\ldots,x_n \,|\, f(a)=0 \,\text{ for all }\, a\in V \right\} , 
\end{align}
the {\em ideal of $V$.} Two affine varieties coincide if and only if their ideals coincide. 
Passing from an affine variety to its ideal is inclusion-reversing in the sense that 
$V\subset W $ if and only if $ I(V)\supset I(W)$. An affine variety is {\em irreducible} if, whenever $V=V_1\cup V_2$ for affine varieties $V_1,V_2$, it follows that $V_1=V$ or $V_2=V$. This condition is often assumed tacitly. A variety is irreducible if and only if its ideal is prime, and, in fact, irreducible varieties in $k^n$ are in one-to-one correspondence with prime ideals in $k[x_1,\ldots,x_n]$. 

The precise relationship between an ideal and the ideal of its variety, can be subtle. For algebraically closed fields, such as the complex numbers, Hilbert's Nullstellensatz relates the ideal of a variety $V(I)$ to the radical ideal $\sqrt{I}=\{ f\in k[x_1,\ldots,x_n] \,|\, \exists m\in \NN\colon \, f^m\in I \}$ of $I$.

\begin{theorem}[The Strong Nullstellensatz]
    Let $k$ be an algebraically closed field and $I\subset k[x_1,\ldots,x_n]$ an ideal. Then
 $I(V(I)) \,=\, \sqrt{I}.$ 
\end{theorem}

The $k$-algebra $k[V]= k[x_1,\ldots,x_n]/I(V)$ is called the {\em coordinate ring of $V$}, and its elements are the {\em regular functions} on~$V$. A {\em morphism} $V\to W$ between two varieties $V,W$ in~$k^n$ is a map that is polynomial in each of its entries, and such maps correspond precisely to $k$-algebra homomorphisms $k[W]\to k[V] $ between the coordinate~rings.

\begin{example}
Let  $f=x_1^2 - 2x_1x_2 + 9x_2^2 - 1\in \RR[x_1,x_2]$. Its variety $V(f)$ is the ellipse in \Cref{fig:affinevarieties} and $\RR[V(f)]=\RR[x_1,x_2]/\langle f\rangle $ its coordinate ring. The regular functions on~$V(f)$ hence are equivalence classes of polynomials: two polynomials are equivalent if they coincide on the ellipse. For instance $[x_1^2-2x_1x_2+9x_2^2]=[1]$ is a constant function on~$V(f)$.
\end{example}

Grothendieck's school of algebraic geometry undertakes a functorial investigation of {\em schemes}. These are a substantial generalization of varieties. Recall that the {\em prime spectrum} of a ring~$R$ 
is the set of all prime ideals of~$R$. Affine schemes are prime spectra of commutative rings, and schemes are certain spaces---namely ``locally ringed spaces''---that can be covered by affine schemes. In this sense, varieties are particular affine schemes over~$k$, namely prime spectra of rings of the form $k[x_1,\ldots,x_n]/I$, with $I$ an ideal in $k[x_1,\ldots,x_n]$. This abstract approach has several advantages. For instance, schemes are defined in the setup of rings instead of only fields, and they encode finer information than varieties. For instance, schemes take multiplicities of points into account: scheme theory does distinguish between the point $\{x=0\}$ and the ``fat point'' $\{x^p=0\}$. 

\subsection{Jet schemes}\label{app:jets}
We now recall basic algebraic notions about jet schemes of affine schemes. Our presentation closely follows~\cite{GS06}.
 
\begin{definition}
 Let $X$ be an affine scheme over a field $k$. An {\em $m$-jet} of $X$ is a $k$-morphism $ \Spec (k\llbracket t \rrbracket /(t^{m +1})) \to X$. The collection of $m$-jets, denoted by $\mathcal{J}_m(X)$,  is a scheme, and is called the {\em $m$-th jet scheme} of~$X$.
  \end{definition}
 The $k$-rational points of $\mathcal{J}_m(X)$ are precisely the $k\llbracket t\rrbracket /(t^{m +1})$-rational points of~$X$. 
% \begin{remark}
Note also that $\mathcal{J}_0(X) \cong X $, and $\mathcal{J}_1(X)$ coincides with the tangent space of~$X$.  In this sense, jet schemes can be seen as a generalization of the tangent space.
%\end{remark}

 Now let $I = \langle f_1 , \ldots ,f_r \rangle \subset k[y_1, \ldots, y_n]$ and  $X=\Spec (k[y_1, \ldots, y_n]/ I)$ its affine scheme. An $m$-jet of $X$ then is a $k$-algebra homomorphism 
\begin{align}
    \varphi\colon \, 
k[y_1,\ldots,y_n]/I\longrightarrow k\llbracket t \rrbracket /(t^{m+1}) \,.
\end{align}
Hence, an $m$-jet $\varphi$ of $X$ is determined by the images of the $y_i$'s.
  For every $i =1,\ldots,n$, $\varphi(y_i)$~is the equivalence class of a power series $y_i(t)$ in~$t$, which we denote as
\begin{align}
y_i(t) \, = \, y_{i, 0} + y_{i,1} t + y_{i, 2}t^2 + \cdots \, . 
\end{align}
For every $s=1,\ldots,r,$ consider
\begin{align}
  f_s((y_1(t), \ldots, y_n(t)) \, = \, f_{s, 0} + f_{s,1} t + f_{s, 2}t^2 + \cdots  \, .
\end{align}
Again, we consider the equivalence class of $f_s$ in $k\llbracket t \rrbracket /(t^{m+1})$.
The coefficient $f_{s,j}$ of $t^j$ is a polynomial in the  $y_{i,\ell}$'s, where $1 \leq i \leq n $ and $\ell \in \NN$. For $m\in \NN$, denote by $\Tilde{I}_m$ the ideal 
\[
 \Tilde{I}_m   \,\coloneqq\, \langle f_{s,j} \,|\, 0 \leq j \leq m ,\,  1 \leq s \leq r \rangle \,\subset\, k [ y_{i, \ell} \,|\, 0 \leq \ell \leq m, \,  1 \leq s \leq r] \, .
\] 
The ideal $\tilde{I}_m$ encodes exactly the condition for $\varphi$ to be well-defined.
The $m$-th jet scheme $\mathcal{J}_m(X)$ of $X$ is the affine scheme determined by~$\tilde{I}_m$.

\begin{example}
    Let $I = \langle y_1 y_2 \rangle$. The $m$-th jet scheme $\mathcal{J}_m(X)$ is determined by the coefficients of $1,t,\ldots,t^m$~in 
    \begin{align*}
    y_1(t) y_2(t) \,=\, \left(y_{1, 0} + y_{1,1} t + y_{1, 2}t^2 + \cdots +y_{1, m} t^m\right)\left( y_{2, 0} + y_{2,1} t + y_{2, 2}t^2 + \cdots +y_{2,m} t^m\right) .
    \end{align*}
    For every $\ell  \leq m$, the coefficient of $t^\ell$ is $\sum_{i = 0}^\ell  y_{1, i} y_{2, \ell - i}$. Hence,
    \begin{align*}
    \Tilde{I}_m \,=\, \left\langle \sum_{i = 0}^\ell  y_{1, i} y_{2, \ell - i} \,|\,  0 \leq \ell \leq m \right\rangle.
    \end{align*}
  Under the identification of the $k$-vector spaces $(k\llbracket t \rrbracket / (t^{m+1}))^2 $ and $k^{2( m + 1)}$, one can show that 
     \begin{align*}
    \mathcal{J}_m(X) &
    \,=\, \bigcup_{  \{ u,v >0 \,|\,u + v \,=\, m +1 \}} \{ (y_{1, u} t^u + \cdots+ y_{1, m} t^m , y_{2, v} t^v + \cdots+ y_{2, m} t^m) \}   
   \\ & \qquad \quad  \cup \,  \left\{\left(y_{1,0}+y_{1,1}t+\cdots + y_{1,m}t^m, 0\right) \right\} \, \cup \, \left\{\left( 0,y_{2,0}+y_{2,1}t+\cdots + y_{2,m}t^m)\right) \right\} \, .
    \end{align*}
This is a decomposition of $\mathcal{J}_m(X)$ into its irreducible components.
\end{example}

We end this appendix by establishing the link to differential algebra.
In fact, the map
\begin{align}\label{eq:isojets}
k \left[\{y_{i,\ell},\, 1 \leq i \leq n, \, 0 \leq \ell \leq m \}\right]/ \Tilde{I}_m \,\stackrel{\cong}{\longrightarrow} \, k \big[y_1^{(\leq m)}, \ldots, y_n^{(\leq m)}\big]/ I^{(
\leq m )} \, , \quad y_{i,\ell} \mapsto \frac{y_i^{(\ell)}}{\ell!} 
\end{align}
is an isomorphism of $k$-algebras.
Hence, the affine scheme $\mathcal{J}_m(X)$ of $m$-jets of $X$ is isomorphic to the affine scheme  $\Spec (k [y_1^{(\leq m)},\ldots, y_n^{(\leq m)}] / I^{(\leq m )} )$.

\bigskip \medskip
{\noindent{\bf Authors' addresses:}

\medskip
\small
\noindent Rida Ait El Manssour, MPI MiS$^{\sharp}$ and IRIF$^{\diamond}$ ({\em current})
\hfill {\tt manssour@irif.fr}

\noindent Anna-Laura Sattelberger, MPI MiS$^{\sharp}$ and KTH$^{\flat}$ \hfill {\tt anna-laura.sattelberger@mis.mpg.de}

\noindent Bertrand Teguia Tabuguia, MPI MiS$^{\sharp}$ \\
\hspace*{2mm} and University of Oxford$^{\circ}$ {\em (current)}   \hfill {\tt bertrand.teguia@cs.ox.ac.uk}

\medskip

\noindent $^{\sharp}$ Max Planck Institute for Mathematics in the Sciences, Inselstra{\ss}e 22, 04103~Leipzig, Germany

\noindent $^{\diamond}$ \hspace*{-1.65mm}  CNRS, IRIF, Université Paris Cité, 8 Pl. Aurélie Nemours, 75013 Paris, France

\noindent $^{\flat}$ Department of Mathematics, KTH Royal Institute of Technology, 100~44 Stockholm, Sweden

\noindent  $^{\circ}$ \hspace*{-1.65mm} Department of Computer Science, Wolfson Building, Parks Road, Oxford OX1 3QD, UK

}
\normalsize

\end{document}